\pdfoutput=1
\documentclass{article}
\usepackage{graphicx}
\usepackage{amsmath,amsfonts,amsthm}
\usepackage{hyperref}
\usepackage{booktabs}
\usepackage{marvosym}

\newcommand{\R}{\mathbb{R}}
\newcommand{\F}{\mathcal{F}}
\newcommand{\Prob}{\mathbb{P}}
\newcommand{\pa}{\partial}
\newcommand{\E}{\mathbb{E}}

\newtheorem{theorem}{Theorem}[section]

\theoremstyle{definition}
\newtheorem{definition}[theorem]{Definition}
\theoremstyle{remark}
\newtheorem{remark}[theorem]{Remark}

\title{Projection methods for stochastic differential equations with conserved quantities}

\author{Weien Zhou\thanks{College of Science, National University of Defense Technology, Changsha 410073, China. (weienzhou@outlook.com).}, Liying Zhang\thanks{School of Mathematical Science, China University of Mining and Technology, Beijing 100083, China. (lyzhang@lsec.cc.ac.cn).}, Jialin Hong\thanks{Institute of Computational Mathematics and Scientific/Engineering Computing, Chinese Academy of Sciences, Beijing, 100190, China. (hjl@lsec.cc.ac.cn).}, Songhe Song\thanks{State Key Laboratory of High Performance Computing, National University of Defense Technology, Changsha 410073, China. (shsong@nudt.edu.cn).}}
\date{}
\begin{document}

\maketitle
\begin{abstract}
	In this paper, we consider the numerical methods preserving single or multiple conserved quantities, and these methods are able to reach high order of strong convergence simultaneously based on some kinds of projection methods. The mean-square convergence orders of these methods under certain conditions are given, which can reach order 1.5 or even 2 according to the supporting methods embedded in the projection step. Finally, three numerical experiments are taken into account to show the superiority of the projection methods.\\
	

\end{abstract}

\section{Introduction}\label{s:intro}
As the process of modeling in science and engineering becomes much more realistic, stochastic differential equations (SDEs) have to be taken into consideration to characterize the stochastic effects. Since most SDEs cannot be solved explicitly, there have been large numbers of works in developing effective and reliable numerical methods for SDEs (see e.g. \cite{burrage2004numerical,kloeden1992numerical,milstein2004stochastic} and references therein).

It is a significant issue whether or not some geometric features of SDEs are preserved in performing reliable numerical methods, especially for long time simulations, which is as important as the deterministic case \cite{hairer2006geometric}. In this aspect, a range of numerical methods have been proposed to preserve different properties of SDEs with special forms (e.g. \cite{hong2007predictor,milstein2002symplectic,misawa2000energy}). Among these properties is the conserved quantity which is intrinsic and essential for some stochastic systems. As a matter of fact, it is crucial to construct numerical methods which can preserve the conserved quantity if the original systems possess.
\cite{misawa2000energy} proposes an energy conservative stochastic difference scheme for one-dimensional stochastic canonical Hamiltonian system and gives the corresponding local errors. In \cite{hong2011discrete}, a discrete gradient method is constructed to preserve single  conserved quantity, which is of mean-square order 1. \cite{chen2014conservative} generalizes the average vector field (AVF) method to stochastic Poisson systems, and also construct first mean-square order method with single conserved quantity. Then \cite{cohen2014energy} proposes a novel conservative method for more general SDEs and analyzes the strong and weak order. In \cite{hong2015preservation}, the authors obtain some conditions for a series of stochastic Runge-Kutta (SRK) methods preserving quadratic conserved quantity.
\cite{Malham2007} presents Lie group integrators for SDEs with
non-commutative vector fields whose solution evolves on a smooth finite-dimensional manifold which can be formed by a conserved quantity.

To the best of our knowledge, the existing  numerical methods cannot preserve general multiple conserved quantities (not limited to quadratic ones) of SDEs. Therefore, we mainly focus on constructing numerical methods preserving the single or multiple conserved quantities.  Another key point is to improve the accuracy of the methods with the conservative property. Based on the two requirements above, we consider the projection methods combined with supporting methods that are common and convenient to solve SDEs, such as Euler, Milstein and high order It\^o-Taylor methods. Then we obtain that the proposed methods share the similar convergence order to the supporting methods that we embed in the projection steps. Therefore, the projection methods are able to  preserve the conserved quantities of SDEs exactly and reach high order of strong convergence simultaneously. The approaches of projection  here are available to any dimensional system greater than one and are not restricted to even-dimensional systems such as stochastic Hamiltonian systems.

The rest of the paper is organized as follows. Sect. \ref{s:preliminary} gives some preliminaries about SDEs with conserved quantities and some necessary notations. In Sect. \ref{s:single}, we outline the basic ideas of projection methods and derive the convergence order in the mean-square sense for stochastic systems with single conserved quantity. In Sect. \ref{s:multi}, we study the corresponding projection methods in the case of systems possessing multiple conserved quantities.  Finally, in Sect. \ref{s:num}, three typical numerical experiments show that our proposed numerical methods are of appropriate convergence order, and have advantages in preserving conserved quantities, which verify the theoretical analysis in previous sections.


Throughout the paper, we will use the following basic notations:
\begin{itemize}
	\item $(\cdot)^\top$: The transpose operator for a vector or matrix.
	\item $(\cdot)^{-1}$: The inverse operator for a invertible square matrix.
	\item $|\cdot|$: The trace norm for vector or matrix defined by $|x| = \sqrt{Tr(x^\top x)}$.
	\item $C_b^k(\R^{d_1}, \R^{d_2})$: The set of $k$ times continuously differential functions $f: \R^{d_1}\rightarrow \R^{d_2}$ with uniformly bounded derivatives up to order k.
\end{itemize}

\section{Preliminary}\label{s:preliminary}
Consider the initial value problem for the general $d$-dimensional autonomous SDE in the sense of Stratonovich:
\begin{equation}\label{e:sde}
\left\{
	\begin{aligned}
	dX(t) &= f\big(X(t)\big)dt + \sum_{r=1}^m g_r\big(X(t)\big) \circ dW_r(t), \quad t\in [0, T],\\
	X(0) &= X_0,
\end{aligned}\right.
\end{equation}
where $X(t)$ is a $d$-dimensional column vector, $W_r(t), r = 1,\dots,m $, are $m$ independent one-dimensional standard Wiener processes defined on a complete filtered probability space $(\Omega , \F, \Prob, \{\F_t\}_{t \geq 0})$ fulfilling the usual conditions, $f$ and $g_r$ are $\R^d$-valued functions satisfying the conditions under which \eqref{e:sde} has a unique solution (see for example \cite{kloeden1992numerical}). $X_0$ is $\F_{0}$-measurable random variable with $\E|X_0|^2<\infty$.

\begin{definition}\label{d:i}
	A differentiable scalar function $I(x)$ is called a conserved quantity (also called an invariant or a first integral) of SDE \eqref{e:sde} if
	\begin{equation}\label{e:def_invrant}
	\nabla I(x)^\top f(x) = 0, \quad \nabla I(x)^\top g_r(x) = 0, \quad r=1,\dots,m,
	\end{equation}
	where $\nabla I(\cdot) = (\frac{\pa I}{\pa x_1}, \dots, \frac{\pa I}{\pa x_d})^\top$ is the gradient of $I(\cdot)$.
\end{definition}

If $X(t)$ is the exact solution of \eqref{e:sde} with the conserved quantity $I(x)$, then we have
\begin{equation}\label{e:dI}
	\begin{aligned}
		dI\big(X(t)\big) = \nabla I\big(X(t)\big)^\top f\big(X(t)\big)dt + \sum_{r=1}^{m} \nabla I\big(X(t)\big)^\top g_r\big(X(t)\big)\circ dW_r(t)
		=0,
	\end{aligned}
\end{equation}
which is obtained by It\^{o}'s formula in Stratonovich sense together with 
\eqref{e:def_invrant}. This simply implies that $I\big(X(t)\big) = I\big(X_0\big)$ holds along the exact solution of  \eqref{e:sde} almost surely. If the initial value $X_0$ is taken to be deterministic, then $I\big(X(t)\big)$ is a deterministic quantity independent of time. From this point of view, the definition above gives a stochastic version of the conserved quantity that is similar to the case of deterministic system. Therefore, we naturally hope that the numerical solutions preserve this property in sample path way as well.

Here are some examples of stochastic systems with conserved quantity. One is the stochastic canonical Hamiltonian system ($d = 2n$) \cite{misawa2000energy}
\begin{equation}\label{e:chamil}
	dX(t) = J^{-1} \nabla H\big(X(t)\big) \left( dt+\sum_{r=1}^m c\circ dW_r(t)\right),
\end{equation}
where $J=\bigl(\begin{smallmatrix}0& I_n\\-I_n&0
\end{smallmatrix}\bigr)$ is a standard $2n$-dimensional symplectic matrix with $n\times n$ identity matrix $I_n$, and $c\in \R$ is a constant parameter. It is an example of \eqref{e:sde} possessing an conserved quantity $H(x)$ which is the Hamiltonian function of \eqref{e:chamil}. A more general stochastic system given by 
\begin{equation}
	dX(t) = J^{-1} \nabla H\big(X(t)\big) dt+ \sum_{r=1}^m J^{-1}\nabla H_r\big(X(t)\big)\circ dW_r(t)
\end{equation}
also has a conserved quantity $I(x)$, if $\{I, H\}=0$ and $\{I, H_r\}=0, r=1,\dots, m$, with Poisson bracket $\{I_1,I_2\}:=\nabla I_1^\top J^{-1} \nabla I_2$ (see \cite{anton2014symplectic,milstein2002numerical} for details). For stochastic Hamiltonian systems, if we denote $X = (P,Q)$ and $X_0 = (p,q)$, the transformation $(p,q)\mapsto (P,Q)$ preserves symplectic structure $dP\wedge dQ = dp\wedge dq$. Many efforts have been made to develop symplectic methods for such systems \cite{milstein2002symplectic,misawa2010symplectic}. However, in general, these methods do not preserve the Hamiltonian $H(x)$ exactly.

Another example is the stochastic Poisson system
\begin{equation}\label{e:poi}
	dX(t) = B\big(X(t)\big)\nabla I\big(X(t)\big)(dt + c\circ dW(t)).
\end{equation}
Here $B(x)$ is a $d \times d$ smooth skew-symmetric matrix-valued function so that $I(x)$ is a conserved quantity of \eqref{e:poi}. \cite{cohen2014energy} mainly studies the energy-preserving scheme for \eqref{e:poi}, which turns out to be mean-square order 1 (for a general multiple noises case see \cite{chen2014conservative} for details).

In the next two sections, we will investigate the projection methods for system \eqref{e:sde} with single and multiple conserved quantities in details, and derive their convergence order in mean-square sense.

\section{Single conserved quantity}\label{s:single}
First we consider the SDE with only one (or only one is known) conserved quantity. Assuming that SDE \eqref{e:sde} possesses a conserved quantity $I : \R^d  \mapsto \R$, then owing to equation \eqref{e:dI}, we have
\begin{equation*}
	X(t)\in \mathcal{M}_{X_0}:=\Big\{x\in\R^d \mid I(x)=I(X_0)\,\Big\}, \quad t\in[0,T], \quad \text{a.s.}
\end{equation*}
This means that $X(t)$ will be confined to the invariant submanifold $\mathcal{M}_{X_0}$ almost surely, which is a direct geometric property for systems possessing a conserved quantity. In addition, SDE \eqref{e:sde}
with a conserved quantity $I(x)$ can be rewritten to another form as well by the following theorem.

\begin{theorem}
	SDE \eqref{e:sde} has the equivalent skew-gradient (SG) form
	\begin{equation}\label{e:SG}
		dX(t) = S\big(X(t)\big)\nabla I\big(X(t)\big)dt + \sum_{r=1}^{m}T_r\big(X(t)\big)\nabla I\big(X(t)\big)\circ dW_r(t),
	\end{equation}
	where $S(x)$ and $T_r(x)$ are skew-symmetric matrices such that $S(x)\nabla I(x)=f(x)$ and $T_r(x)\nabla I(x)=g_r(x)$, $r=1,\dots, m$.
\end{theorem}
The proof of this theorem is constructive and can be found in \cite{chen2014conservative}. In general, $S(x)$ and $T_r(x)$ are not unique. For instance, if $\nabla I(X(t))\neq 0$, one simple choice is the default formula:
\begin{equation}\label{e:default_formula}
	\begin{aligned}
		S(x)   &= \frac{f(x)\nabla I(x)^\top - \nabla I(x) f(x)^\top}{|\nabla I(x)|^2},\\
		T_r(x) &= \frac{g_r(x)\nabla I(x)^\top - \nabla I(x) g_r(x)^\top}{|\nabla I(x)|^2},\quad  r=1,\dots,m.
	\end{aligned}
\end{equation}

Based on the particular SG formula \eqref{e:SG} of \eqref{e:sde}, we are able to replace the gradient in \eqref{e:SG} with the discrete gradient, which leads to the discrete gradient method (see \cite{hong2011discrete} for details). It is beneficial to approximate the solution to SDE \eqref{e:sde} and preserve the general conserved quantity $I(x)$ simultaneously. On the other hand, another kind of approaches to achieve this goal are the projection methods, which we will apply to preserve the conserved quantity in the stochastic case.

The basic idea of projection methods is to combine an arbitrary one-step approximation $\widehat{X}_{t,x}$ starting at $x$ together with a projection $P$ onto the invariant submanifold $\mathcal{M}_x$ in every step. Thus the procedures of the methods at each step are:
\begin{enumerate}
	\item Compute the one-step approximation $\widehat{X}_{t,x}$.
	\item Compute $\lambda\in \R$, for $\bar{X}_{t,x} = \widehat{X}_{t,x} + \Phi\lambda$ \, s.t. \, $I(\bar{X}_{t,x}) = I(x)$.
\end{enumerate}
Here the vector $\Phi \in \R^d$ defines the direction of the projection, and $\lambda$ is a scalar chosen such that the new approximation $\bar{X}_{t,x}$ belongs to the invariant submanifold $\mathcal{M}_x$ properly. The general idea of the procedures is shown in Fig. \ref{f:proj}.

\begin{figure}[ht]
	\centering
	\includegraphics{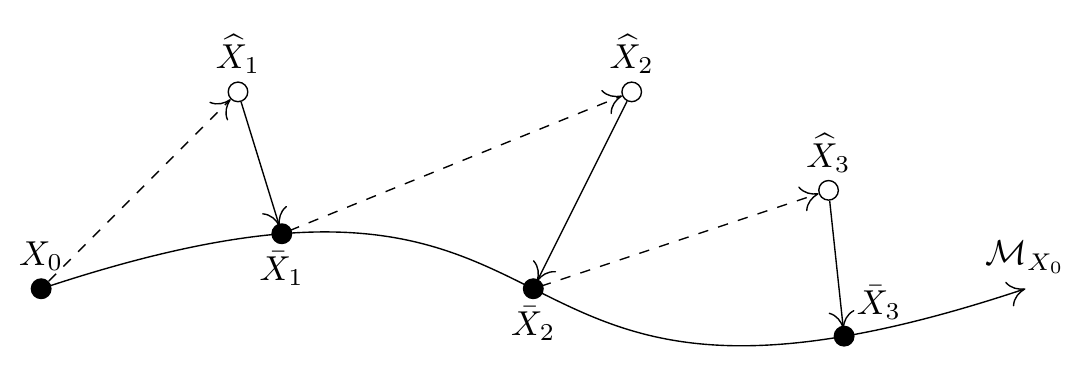}
	\caption{Basic idea of the projection methods.}
	\label{f:proj}
\end{figure}

In fact the projection direction $\Phi$ is not unique here, and the standard orthogonal projection chooses $\nabla I(\bar{X}_{t,x})$ as $\Phi(x,\bar{X}_{t,x})$ \cite[Chap. IV]{hairer2006geometric}. In order to reduce computational cost we just use $\nabla I(\widehat{X}_{t,x})$ or $\nabla I(x)$ to replace $\nabla I(\bar{X}_{t,x})$. So we combine the relationship and get $\bar{X}_{t,x}$ above by solving a non-linear equation with respect to $\lambda$. In fact, we find that $\lambda$ approaches 0 at each step, so we just set $\lambda_0 = 0$ in Newton iteration to solve it.

Since we are mainly interested in the geometric features of the numerical solutions in path-wise way, the mean-square convergence is considered here. Throughout the paper, equidistant time step-size $h$ will be used in time discretization $\{0=t_0,\dots,t_N=T\}$. In the sequel, let $X_k$ be the numerical approximation of SDE \eqref{e:sde} at time $t_k$, $k=0,\dots,N$.

\begin{definition}
	A numerical method $X_k$ is said to have mean-square order $p$, if
	\begin{equation}
		\big(\E|X_k - X(t_k)|^2 \big)^{\frac{1}{2}} = O(h^p),
		\quad k = 0,\dots,N.
	\end{equation}
\end{definition}

Next, we consider the mean-square convergence of the projection method with a particular supporting one-step method $\widehat{X}$. Suppose that the supporting one-step method $\widehat{X}$ is of mean-square order $p$, then we will prove that the projection method $\bar{X}$ with it has the same mean-square order under certain conditions. Here we will make use of the following theorems proposed in \cite{milstein2004stochastic}.

\begin{theorem}\label{t:p1p2}
	\cite{milstein2004stochastic} Suppose that the one-step approximation $\bar{X}_{t,x}(t+h)$ has order of accuracy $p+1$ for the mean deviation and order of accuracy $p+\frac{1}{2}$ for the mean-square deviation; more precisely, for any $0\leq t \leq T-h$, $x\in \R^d$ the following inequalities hold:
	\begin{equation}\label{c:p1}
		|\E(X_{t,x}(t+h) - \bar{X}_{t,x}(t+h))| \leq K(1+|x|^2)^{1/2} h^{p+1},
	\end{equation}
	\begin{equation}\label{c:p2}
		\big(\E|X_{t,x}(t+h) - \bar{X}_{t,x}(t+h)|^2\big)^{1/2} \leq K(1+|x|^2)^{1/2} h^{p+\frac{1}{2}},
	\end{equation}
	Also let
	\begin{equation}
		p \geq \frac{1}{2}.
	\end{equation}
	Then for any $N$ and $k=0,1,\dots N$ the following inequality holds:
	\begin{equation}
		(\E|X_{t_0,X_0}(t_k)-\bar{X}_{t_0,X_0}(t_k)|^2)^{1/2} = O(h^{p}),
	\end{equation}
	i.e. the mean-square order of accuracy of the method constructed using the one-step approximating $\bar{X}_{t,x}$ is $p$.
\end{theorem}

\begin{theorem}\label{t:pp1p2}
	\cite{milstein2004stochastic} Let the one-step approximation $\bar{X}_{t,x}$ satisfy the conditions of Theorem \ref{t:p1p2}. Suppose that another approximation $\widehat{X}_{t,x}$ is such that
	\begin{equation}\label{e:pp1}
		|\E(\widehat{X}_{t,x} - \bar{X}_{t,x})| = O(h^{p+1}),
	\end{equation}
	\begin{equation}
		(\E|\widehat{X}_{t,x} - \bar{X}_{t,x}|^2)^{1/2} = O(h^{p+\frac{1}{2}})
	\end{equation}
	with the same $p$ for $\bar{X}_{t,x}$ in Theorem \ref{t:p1p2}. Then the method based on the one-step approximation $\widehat{X}_{t,x}$ has the same mean-square order of accuracy as the method based on $\bar{X}_{t,x}$, i.e. its mean-square order is equal to $p$ as well.
\end{theorem}

Usually the increments $\Delta W_r(h):=W_r(t_n+h)-W_r(t_n)$, $r=1,\dots,m$, in common methods are represented by $\sqrt{h}\xi_r$, and $\xi_r$ are independent $N(0,1)$-distributed random variables. Because of the unbounded property of the standard Gaussian random variable, the independent Gaussian random variables $\Delta W_r$ used here should be truncated as $\Delta \widehat{W_r}(h):=\sqrt{h}\zeta_h^r$
\begin{equation*}
	\zeta_h^r =
	\begin{cases}
		\xi^r, & \,|\xi^r|\leq A_h \\
		A_h,   & \,\xi^r > A_h\\
		-A_h,  & \,\xi^r < -A_h
	\end{cases}
\end{equation*}
with $A_h:=\sqrt{2k|ln(h)|}$, and $k$ is an positive integer (see \cite{milstein2002numerical} for details). So the following inequality holds
\begin{equation*}
	\E(\xi^r-\zeta^r_h)^2\leq h^k.
\end{equation*}
From now on, we assume the  increments in the supporting method $\widehat{X}$ are replaced by the truncated forms $\Delta\widehat{W}(h)$ with proper integer $k$. The truncated methods will remain the same mean-square order due to Theorem \ref{t:pp1p2}.

Based on above preparation works, we present our first main result in this section:
\begin{theorem}\label{t:single}
	Assume that a supporting method $\widehat{X}$ applying to \eqref{e:sde} satisfies \eqref{c:p1} and \eqref{c:p2} with $p$, and the matrix functions $S,T_r \in C_b^2(\R^d,\R^{d\times d})$, $r=1,\dots,m$, in the equivalent SG form \eqref{e:SG}. Moreover, assume that $\nabla I$ satisfies global Lipschitz condition and has uniformly bounded derivatives up to order 2, $|\nabla I|$ has a positive lower bound and $\big(|\nabla I|^2\big)^{-1}$ has bounded derivative near the invariant submanifold. Then the projection method $\bar{X}$ with the supporting method $\widehat{X}$  has mean-square order p as well.
\end{theorem}

\begin{proof}
	For simplicity of exposition, we assume that $m=1$ and the supporting method $\widehat{X}$ only has the Wiener increment $\Delta \widehat{W}(h)$. To lighten the notations, we use a generic positive constant $C$ here and throughout this paper, which is independent of $h$ but may be different from line to line. Firstly, we rewrite the one-step projection method in the form
	\begin{equation}
		\bar{X}_{t,x} = \widehat{X}_{t,x} + \lambda\nabla I(\widehat{X}_{t,x}),
	\end{equation}
	where $\bar{X}_{t,x}$ and $\widehat{X}_{t,x}$ start at $x$ in time $t$. In order to simplify notations, the subscripts of them are omitted hereafter.
	
	Define a function $F\big(\lambda, h, \Delta \widehat{W}(h)\big):= I\big(\bar{X}\big) - I(x)$, then we have
	\[F\big(0,0,\Delta \widehat{W}(0)\big) = I(x) - I(x) = 0,\]
	\[\left.\frac{\pa{F}}{\pa{\lambda}} \right|_{\lambda=0, h=0} = \nabla I(x)^\top \nabla I(x) = |\nabla I(x)|^2 \neq 0.\]
	Thus, by the implicit function theorem, there exist a sufficient small $h_0>0$ and a unique function $\lambda(h,\Delta \widehat{W}(h))$, such that $\lambda(0,\Delta \widehat{W}(0))=0$ and  $F\big(\lambda, h, \Delta \widehat{W}(h)\big) = 0$, for $h\in(0,h_0].$
	
	Expanding $F$ with respect to $\lambda$ at 0, we obtain
	\begin{align*}
		0 &= F\big(0,h,\Delta \widehat{W}(h)\big) + \left.\frac{\pa{F}}{\pa{\lambda}}\right|_{\lambda=0} \lambda + \frac{1}{2} \left.\frac{\pa^2{F}}{\pa{\lambda^2}}\right|_{\lambda=\lambda_1} \lambda^2 \\
		&= I(\widehat{X}) - I(x) + \nabla I(\widehat{X})^\top \nabla I(\widehat{X})\lambda + \frac{1}{2} \nabla I(\widehat{X})^\top \nabla^2 I(\theta)\nabla I(\widehat{X}) \lambda^2,
	\end{align*}
	where $\lambda_1\in [0,\lambda]$, $\theta = \widehat{X}+\nabla I(\widehat{X})\lambda_1$, and $\nabla^2 I(\cdot)$ denotes the Hessian matrix of $I(\cdot)$.
	
	Since $\nabla I(\widehat{X})^\top \nabla I(\widehat{X}) = |I(\widehat{X})|^2 \neq 0$ by assumption, we have
	\begin{equation}\label{e:lambda}
		\begin{aligned}
			\lambda &= -\Big(\nabla I(\widehat{X})^\top \nabla I(\widehat{X})\Big)^{-1}  \Big(I(\widehat{X}) - I(x)\Big) \\
			& \quad {}-\frac{1}{2}\Big(\nabla I(\widehat{X})^\top \nabla I(\widehat{X})\Big)^{-1} \nabla I(\widehat{X})^\top \nabla^2 I(\theta)\nabla I(\widehat{X}) \lambda^2.
		\end{aligned}
	\end{equation}
	
	In addition, suppose that $X := X_{t,x}(t+h)$ is the solution of the SDE with initial value $x$, we thus have $I(X) = I(x)$, and
	\begin{equation}\label{e:I(hat)-I(x)}
		\begin{aligned}
			&\phantom{==} I(\widehat{X}) - I(x) \\
			&= I(\widehat{X}) - I(X) \\
			&= \nabla I(X)^\top(\widehat{X}-X) + (\widehat{X}-X)^\top\nabla^2 I(\theta_1)(\widehat{X}-X)\\
			&= \nabla I(x)^\top(\widehat{X}-X) + (\widehat{X}-X)^\top\nabla^2 I(\theta_1)(\widehat{X}-X) + (X-x)^\top\nabla^2 I(\theta_2)(\widehat{X}-X)
		\end{aligned}
	\end{equation}
	where $\theta_1$ is on the segment connecting $X$ and  $\widehat{X}$, and $\theta_2$ is on the segment connecting $x$ and $X$.
	
	Substituting $I(\widehat{X})-I(x)$ into \eqref{e:lambda} and using the boundedness assumptions in the theorem, we get
	\begin{equation}\label{e:lambda2}
		\begin{aligned}
			|\lambda|^2 &\leq C |\widehat{X}-X|^2 + C |\lambda|^4\\
			&\leq C |\widehat{X}-X|^2 + \frac{1}{2} |\lambda|^2.
		\end{aligned}
	\end{equation}
	The second inequality in \eqref{e:lambda2} is due to the fact
	that $\lambda\rightarrow 0$ as $h\rightarrow 0$. So when $h$ is sufficiently small, $C |\lambda|^4\leq \frac{1}{2} |\lambda|^2$ holds. Then
	\begin{equation}\label{e:E|lambda^2|}
		\begin{aligned}
			\E|\lambda|^2 &\leq C \E|\widehat{X}-X|^2 = O(h^{2p+1}),
		\end{aligned}
	\end{equation}
	because we assume that $\widehat{X}$ satisfies \eqref{c:p2} in Theorem \ref{t:p1p2} with $p$.
	Furthermore, to get the desired result, we must estimate $|\E(\lambda)|$. We continue to expand $\lambda$
	\begin{equation}
		\lambda = -\Big(\nabla I(x)^\top \nabla I(x)\Big)^{-1} \Big(I(\widehat{X})-I(x)\Big) - R_1,
	\end{equation}
	with 
	\begin{align*}
		R_1 &= \Big(\big(\nabla I(\theta_3)^\top \nabla I(\theta_3)\big)^{-1}\Big)'(\widehat{X}-X) \Big(I(\widehat{X})-I(x)\Big) \\
		& \quad + \frac{1}{2}\Big(\nabla I(\widehat{X})^\top \nabla I(\widehat{X})\Big)^{-1} \nabla I(\widehat{X})^\top \nabla^2 I(\theta)\nabla I(\widehat{X}) \lambda^2,
	\end{align*}
	where $\theta_3$ is on the segment connecting $x$ and  $\widehat{X}$. Therefore,
	\begin{equation}\label{e:E|lambda|}
		|\E(\lambda)| \leq |(\nabla I(x)^\top \nabla I(x))^{-1}|\cdot |\E(I(\widehat{X})-I(x))| + |\E(R_1)|.
	\end{equation}
	
	Note that, since $\E(\widehat{X} - X)=O(h^{p+1})$, $\E(\widehat{X} - X)^2 = O(h^{2p+1})$ and $\E(X - x)^2=O(h)$, this yelds
	\begin{equation}\label{e:E(I-I)}
		\begin{aligned}
			\E\big(I(\widehat{X}) - I(x)\big) &= \nabla I(x)^\top \E(\widehat{X}-X) + \E\Big((\widehat{X}-X)^\top\nabla^2 I(\theta_1)(\widehat{X}-X)\Big) \\
			&\quad + \E\Big((X-x)^\top\nabla^2 I(\theta_2)(\widehat{X}-X)\Big)\\
			&\leq C \E(\widehat{X}-X) + C \E(\widehat{X}-X)^2 + C \Big(\E(X-x)^2\Big)^{1/2} \Big(\E(\widehat{X}-X)^2\Big)^{1/2}\\
			& = O(h^{p+1})
		\end{aligned}
	\end{equation}
	and
	\begin{equation}\label{e:ER1}
		\E(R_1)\leq C \Big(\E(\widehat{X}-X)^2\Big)^{1/2} \Big(\E\big(I(\widehat{X})-I(x)\big)^2\Big)^{1/2} + C \E(\lambda^2) = O(h^{2p+1})
	\end{equation}
	by the Cauchy-Schwartz inequality.
	Substituting \eqref{e:E(I-I)} and \eqref{e:ER1} into \eqref{e:E|lambda|}, we deduce that
	\begin{equation}\label{e:|E(lambda)|}
	\vert \E(\lambda)\rvert = O(h^{p+1}).
	\end{equation}
	Now we can compare $\bar{X}$ and $\widehat{X}$ by the use of \eqref{e:E|lambda^2|} and \eqref{e:E|lambda|}. Since
	\begin{align*}
		\bar{X}-\widehat{X} &= \nabla I(\widehat{X}) \lambda = \nabla I(x) \lambda + \nabla^2 I(\theta_4) (\widehat{X}-x) \lambda,
	\end{align*}
	where $\theta_4$ is on the segment connecting $\widehat{X}$ and  $x$,
	we have 
	\begin{align*}
		\E|\bar{X}-\widehat{X}|^2 &\leq 2 |\nabla I(x)|^2 \E|\lambda|^2 + 2|\nabla^2 I(\theta_4)| \E|(\widehat{X}-x)\lambda|^2\\
		&\leq C \E|\lambda|^2 + C \E|\widehat{X}-x|^2 \E|\lambda|^2\\
		&= O(h^{2p+1})
	\end{align*}
	and
	\begin{align*}
		|\E(\bar{X}-\widehat{X})| &\leq |\nabla I (x)| \cdot|\E(\lambda)| + \E|\nabla^2 I(\theta_4) (\widehat{X}-x)\lambda|\\
		&\leq C |\E(\lambda)| + C (\E|\widehat{X}-x|^2)^{\frac{1}{2}} (\E|\lambda|^2)^{\frac{1}{2}}\\
		& = O(h^{p+1}).
	\end{align*}
	That is to say, $|\E(\bar{X}-\widehat{X})|=O(h^{p+1})$ and $(\E|\bar{X}-\widehat{X}|^2)^{\frac{1}{2}} = O(h^{p+\frac{1}{2}})$. Finally, we attain that the projection method $\bar{X}$ has mean-square order $p$ just as the supporting method $\widehat{X}$ by applying Theorem \ref{t:pp1p2}. The proof is therefore completed.\qed
\end{proof}

\begin{remark}
	We can also obtain the same result by choosing $\nabla I(x)$ instead of $\nabla I(\widehat{X})$ as the projection direction.
\end{remark}


Here we consider some common one-step supporting methods which can be applied to the projection methods:

1. The Euler-Maruyama method of mean-square order 0.5 \cite{kloeden1992numerical}
\begin{equation}\label{m:euler}
\widehat{X} = x + h\left(f(x) + \frac{1}{2}\sum_{r=1}^{m}\frac{\pa g_r}{\pa x} g_r(x)\right) + \sum_{r=1}^m g_r(x)\Delta\widehat{ W}_r(h),
\end{equation}
which is gained by converting \eqref{e:sde} into the equivalent It\^{o} sense first.

2. The Milstein method of mean-square order 1 with commutative noises ($\Lambda_i g_r = \Lambda_r g_i$) \cite{milstein2004stochastic}
\begin{equation}\label{m:mil}
		\begin{aligned}
			\widehat{X} &= x + hf(x) + \sum_{r=1}^m g_r(x)\Delta \widehat{W}_r + \sum_{i=1}^{m-1}\sum_{r=i+1}^{m}\Lambda	a_i g_r(x) \Delta \widehat{W}_i(h) \Delta \widehat{W}_r(h) +\frac{1}{2}\sum_{r=1}^m \Lambda_r g_r(x)\big(\Delta \widehat{W}_r(h)\big)^2,
			\end{aligned}
\end{equation}
where the operator $\Lambda_i g_r(x):= \Big(g_r(x)\Big)'g_i(x)$.

3. The mid-point method of mean-square order 1 with commutative noises \cite{hong2015preservation,milstein2002numerical}
\begin{equation}\label{m:mid}
	\widehat{X} = x + f(\frac{x+\widehat{X}}{2})h + \sum_{r=1}^m g_r(\frac{x+\widehat{X}}{2}) \Delta \widehat{W}_r(h).
\end{equation}

They all can be easily applied as the supporting methods since they are just in need of simulating the Wiener increments $\Delta \widehat{W}_r$ $(r=1,\dots,m)$ at each step.

Notice that the mid-point method \eqref{m:mid} is implicit, and is a symplectic method for stochastic Hamiltonian systems which automatically preserves quadratic conversed quantity\cite{hong2015preservation}. Now applying the mid-point method \eqref{m:mid} as the supporting method in the projection method, and using the property of a kind of discrete gradient $\bar{\nabla} I(x,\bar{X})$:
\begin{equation}\label{dg_property}
\bar{\nabla}I(x,\bar{X})^\top \cdot (\bar{X}-x) = I(\bar{X})-I(x),
\end{equation}
we can obtain that
\begin{equation}\label{m:dg}
	\bar{X} = x + h \bar{S} \bar{\nabla}I(x,\bar{X})  + \sum_{r=1}^m \Delta\widehat{W}_r(h) \bar{T}_r \bar{\nabla}I(x,\bar{X}),
\end{equation}
where $\bar{S}$ and $\bar{T_r}$ are both $d \times d$ skew-symmetric matrices defined by
\begin{equation}
	\begin{aligned}
		\bar{S} &= \frac{f(\frac{x+\widehat{X}}{2})\nabla I(\widehat{X})^\top - \nabla I(\widehat{X}) f(\frac{x+\widehat{X}}{2})^\top}{\nabla I(\widehat{X})^\top \bar{\nabla}I(x,\bar{X})},\\
		\bar{T}_r &= \frac{g_r(\frac{x+\widehat{X}}{2})\nabla I(\widehat{X})^\top - \nabla I(\widehat{X}) g_r(\frac{x+\widehat{X}}{2})^\top}{\nabla I(\widehat{X})^\top \bar{\nabla}I(x,\bar{X})},\, r=1,\dots,m.
	\end{aligned}
\end{equation}
From this point of view, we can regard method \eqref{m:dg} as another form of the discrete gradient method which is an extension in the case of stochastic systems \cite{norton2013projection}.

\section{Multiple conserved quantities}\label{s:multi}

In this section, we consider projection methods for system \eqref{e:sde} with multiple conserved quantities. Below is the definition of the system with multiple conserved quantities.
\begin{definition}\label{d:multiple}
	System \eqref{e:sde} possesses $l$  independent conserved quantities $I^i(x)$, $i=1,\dots ,l$, if
	\begin{equation}
		\big(\nabla I^i(x)\big)^\top f(x) = 0 \quad\text{and}\quad\big(\nabla I^i(x)\big)^\top g_r(x) = 0,
		\quad r=1,\dots,m; \quad i=1,\dots,l.
	\end{equation}
\end{definition}

In other words, if we define vector $\mathbf{I}(x) := \big(I^1(x),\dots,I^l(x)\big)^\top$, then 
\[\mathbf{I}'(x)f(x)=\mathbf{I}'(x)g_r(x)=\mathbf{0}, \quad r=1,\dots,m,\]
where $\mathbf{I}'(x)$ is the Jacobian matrix of $\mathbf{I}(x)$. So this definition is consistent with Definition \ref{d:i}.

If system \eqref{e:sde} possesses $l$ conserved quantities $I^i(x)$, $i=1,\dots ,l$, then by \eqref{e:dI} we have
\begin{equation*}
	X(t)\in \mathcal{M}_{X_0}:=\Big\{x\in\R^d \mid I^i(x)=I^i(X_0),
	i=1,\dots,l\Big\}, \quad t\in[0,T], \quad \text{a.s.},
\end{equation*}
which implies that the solution $X(t)$ of this system will be confined to the invariant submanifold $\mathcal{M}_{X_0}$ generated by $I^i(x)$, $i=1,\dots,l$.

For system \eqref{e:sde} with multiple conserved quantities, it is difficult to construct numerical methods approximating the exact solution while preserving these quantities simultaneously. To overcome this difficulty, we also couple an common supporting method with a projection step to construct the numerical method in stochastic case, then the resulted numerical approximation will stay at the proper invariant submanifold as the exact solution in every sample path.

Suppose that we have a supporting one-step method $\widehat{X}_{t,x}$, the projection method here shares similarities with the one in Sect. \ref{s:single}:
\begin{enumerate}
	\item Compute the one-step approximation $\widehat{X}_{t,x}$.
	\item Compute $\mathbf{\lambda}\in \R^l$ for $\bar{X}_{x,t} = \widehat{X}_{t,x} + \Phi\lambda$,\, s.t.\, $\mathbf{I}(\bar{X}_{x,t}) = \mathbf{I}(x)$.
\end{enumerate}
Here the matrix $\Phi \in \R^{d\times l}$ defines the direction of the projection, and $\mathbf{\lambda}$ is a $l$-dimensional vector chosen such that $\bar{X}_{t,x}$ belongs to the invariant manifold $\mathcal{M}_{X_0}$. In fact $\Phi$ is not unique, and here we choose  $\Phi=\bigl(\mathbf{I}'(\widehat{X}_{t,x})\bigr)^\top$, which is transpose of the Jacobian matrix of $\mathbf{I}(\cdot)$ at $\widehat{X}_{x,t}$.

Considering the mean-square order of convergence of the projection methods for system with multiple conserved quantities leads to the following theorem, which is the main result of this section.

\begin{theorem}
	Suppose that system \eqref{e:sde} possesses $l$ independent conserved quantities $I^i(x),i=1,\dots,l$.
	Also assume that a supporting method $\widehat{X}$ applying to \eqref{e:sde} satisfies \eqref{c:p1} and \eqref{c:p2} in Theorem \ref{t:p1p2} with $p$ and has mean-square order $p$, and the matrix functions $S^i,T_r^i \in C_b^2(\R^d,\R^{d\times d})$, $r=1,\dots,m$, in the equivalent SG forms \eqref{e:SG}. Assume that $\nabla I^i$ satisfy global Lipschitz condition and have uniformly bounded derivatives up to order 2, $|\nabla I^i|$ have a positive lower bound and $\big(|\nabla I^i|^2\big)^{-1}$ has bounded derivative near the invariant manifold. Then the projection method $\bar{X}$ using the supporting mehtod $\widehat{X}$ also has mean-square order $p$.
\end{theorem}

The proof of this theorem is similar to that of Theorem \ref{t:single}. In order to avoid repetition, we omit it here.

To realize the projection method above, we have to solve a $l$-dimensional nonlinear system, which can be implemented by some iterative algorithms such as the Newton method as well. In addition, if the Jacobian matrix $\mathbf{I}'(x)$ is unavailable or too expensive to compute at every iteration, some kinds of Quasi-Newton methods can be used \cite{dennis1977quasi}.
\section{Numerical experiments}\label{s:num}
In this section, three typical examples are given to show the conserved quantities along numerical solutions and mean-square convergence order for our proposed projection methods in Sect. \ref{s:single} and \ref{s:multi}. Here we consider the following systems:
\begin{equation}\label{e:f=cg}
	dX(t) = f\big(X(t)\big) \left(dt + \sum_{r=1}^{m} c_r \circ dW_r(t)\right)
\end{equation}
with real-valued constants $c_r$, $r=1,\cdots m$, representing the intensity of each noise, then we select the following general popular methods as the supporting ones embedded in the projection methods:
\begin{enumerate}
	\item Euler method.
	\item Milstein method.
	\item Mid-point method (Mid).
	\item Order 1.5 strong Taylor method (T3/2).
	\item Order 2 strong Taylor method (T2).
\end{enumerate}

Details of the last two methods with some more terms according to stochastic Taylor expansion can be found in \cite{kloeden1992numerical,milstein2004stochastic}.
The corresponding projection methods are denoted by EulerP, MilsteinP, MidP, T3/2P, and T2P for short respectively.
In addition, it is quite convenient to perform them  for  \eqref{e:f=cg} because they are just in need  of modeling the Wiener increments $\Delta \widehat{W}_r(h)$ at each time step.

\subsection{Example 1. Kubo Oscillator}
First of all, we consider a stochastic harmonic oscillator with one multiplicative noise in the sense of Stratonovich defined by
\begin{equation}\label{kubo}
	\left\{
	\begin{aligned}
		dX_1(t) &= -a X_2(t) dt - \sigma X_2(t) \circ dW(t),\\
		dX_2(t) &= \phantom{-} a X_1(t) dt + \sigma X_1(t) \circ dW(t),
	\end{aligned}\right.
\end{equation}
where $a$ and $\sigma$ are real-valued parameters.
This system is also known as the Kubo oscillator and appears in nuclear magnetic resonance, the theory of optical line shapes, and also molecular spectroscopy \cite{cohen2012numerical}. Moreover, system \eqref{kubo} is a stochastic Hamiltonian system \eqref{e:chamil} with Hamiltonian function $I(x,y) = \frac{1}{2}(x^2 +y^2)$ as its conserved quantity \cite{milstein2002numerical}.

To test the convergence order of the projection methods, we experiment here the methods mentioned above as corresponding supporting methods in Sect. \ref{s:single}. Let $a=\sigma=1$, $\Big(X_1(0), X_2(0)\Big)=(1,0)$ be the parameters and initial value in \eqref{kubo} respectively.

\begin{table}[]	
	\caption{Mean-square errors of Kubo oscillator}
	\label{table:kubo}
	\begin{tabular}{@{}lccccccr@{}}
		\toprule
		        $h$         & $2^{-3}$ & $2^{-4}$ & $2^{-5}$ & $2^{-6}$ & $2^{-7}$ & $2^{-8}$ & Order \\ \midrule
		       Euler        & 3.11E-01 & 1.92E-01 & 1.26E-01 & 8.68E-02 & 6.31E-02 & 4.49E-02 & 0.55  \\ \midrule
		      EulerP        & 1.52E-01 & 6.99E-02 & 3.98E-02 & 1.85E-02 & 8.79E-03 & 4.66E-03 & 1.01  \\ \midrule
		     Milstein       & 2.20E-01 & 1.01E-01 & 4.98E-02 & 2.41E-02 & 1.18E-02 & 5.84E-03 & 1.04  \\ \midrule
		MilsteinP\quad\quad & 1.01E-01 & 4.77E-02 & 2.46E-02 & 1.19E-02 & 5.79E-03 & 3.09E-03 & 1.01  \\ \midrule
		        Mid         & 5.27E-02 & 2.50E-02 & 1.28E-02 & 6.31E-03 & 2.99E-03 & 1.56E-03 & 1.02  \\ \midrule
		       T3/2         & 3.85E-02 & 1.29E-02 & 4.40E-03 & 1.48E-03 & 5.78E-04 & 2.03E-04 & 1.51  \\ \midrule
		       T3/2P        & 3.36E-02 & 1.12E-02 & 3.66E-03 & 1.33E-03 & 5.22E-04 & 1.83E-04 & 1.50  \\ \midrule
		        T2          & 1.59E-02 & 3.73E-03 & 9.39E-04 & 2.29E-04 & 5.70E-05 & 1.49E-05 & 2.01  \\ \midrule
		        T2P         & 1.38E-02 & 3.08E-03 & 7.30E-04 & 1.73E-04 & 4.43E-05 & 1.13E-05 & 2.05  \\ \bottomrule
	\end{tabular}
\end{table}

\begin{figure}[htbp]
	\centering
	\includegraphics[width=0.7\textwidth]{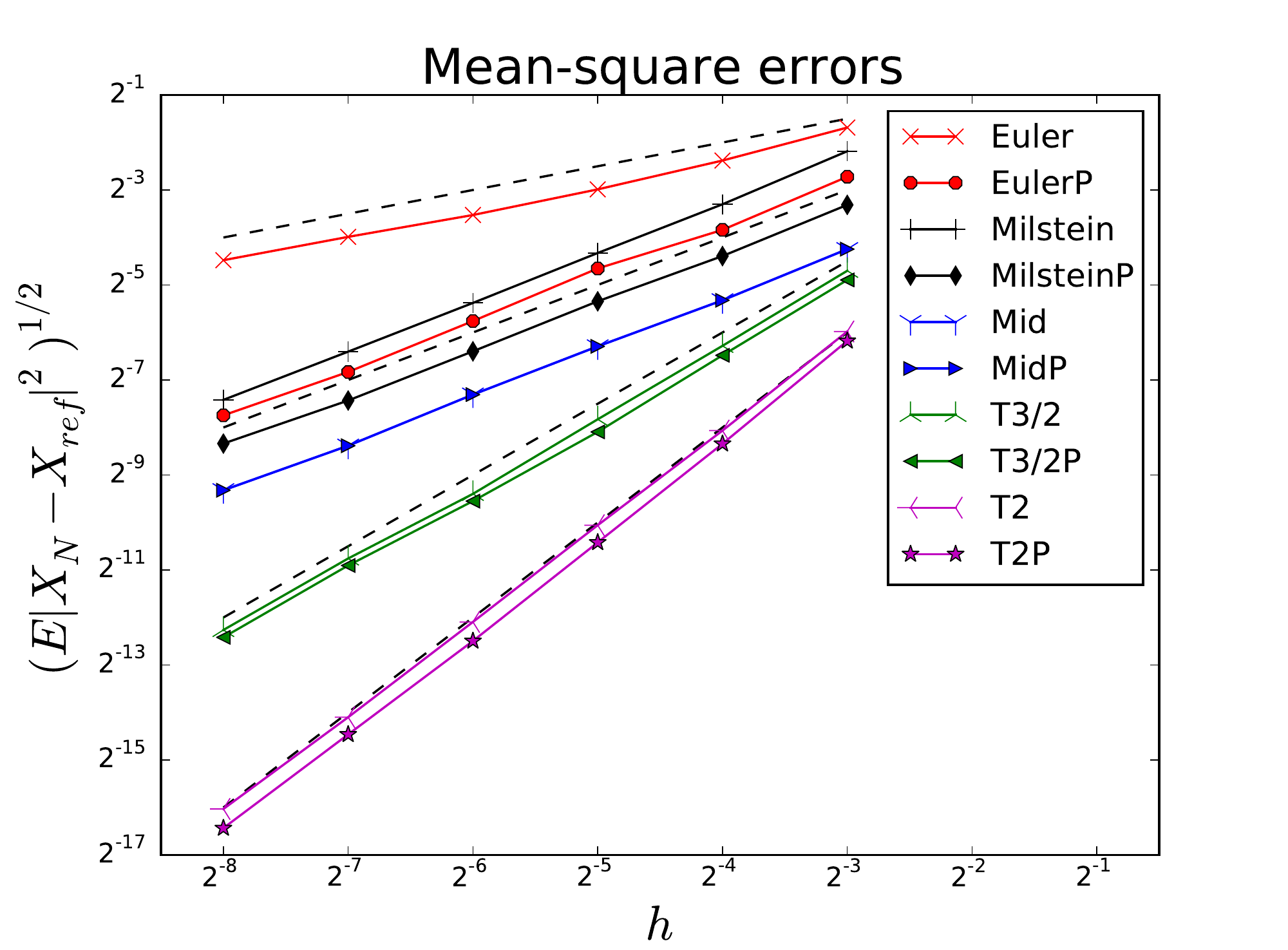}
	\caption{Mean-square errors at $T=1$ of these methods for Kubo oscillator with $a = 1$, $\sigma = 1$, $(X_1(0) , X_2(0))=(1 , 0)$. The dashed reference lines have slopes 0.5, 1, 1.5 and 2 respectively.}\label{f:kubo_order}
\end{figure}

Table \ref{table:kubo} displays the mean-square errors at the final step $T=1$ of these methods, where the last column lists the convergence order calculated by method of linear least square fitting \cite{higham2001algorithmic}. Also, they are represented in Fig. \ref{f:kubo_order} graphically. Here, the reference solution is computed using the T2 method and step-size $h_{\mathrm{ref}}=2^{-14}$ and the expectation is realized using the average of $M=10000$ independent sample paths. Because system \eqref{kubo} is linear and has structure \eqref{e:f=cg}, the methods we consider here just need Wiener increments $\Delta \widehat{W}(h)$ at each step. Note that these methods have the mean-square orders corresponding to the result of Sect. \ref{s:single}. Although the Euler method has mean-square order 0.5, it seems to increase to order 1 after projection. However, this result can be explained from another aspect as below.

Rewrite \eqref{kubo} to 
\begin{equation}\label{e:kubo_dg}
	dX = S\nabla I(X)dt + T\nabla I(X) \circ dW,
\end{equation}
where constant matrices $S$ and $T$ are 
\begin{equation}\label{e:ST}
	\begin{aligned}
		S &= \frac{f(x)\nabla I(x)^\top-\nabla I(x)f(x)^\top}{\nabla I(x)^\top \nabla I(x)}=
		\begin{pmatrix}
			0 & \quad -a\\
			a & \quad 0
		\end{pmatrix},\\
		T &= \frac{g(x)\nabla I(x)^\top-\nabla I(x)g(x)^\top}{\nabla I(x)^\top \nabla I(x)}=
		\begin{pmatrix}
			0  & \quad -\sigma\\
			\sigma & \quad 0
		\end{pmatrix}.
	\end{aligned}
\end{equation}

Applying the discrete gradient method to \eqref{e:kubo_dg}, we get
\begin{equation}\label{e:dg_kubo}
	\bar{X} = x + h S \bar{\nabla}I(x,X) + \Delta\widehat{W}(h) T \bar{\nabla}I(x,X)
\end{equation}
where $\bar{\nabla}I(x,X)$ is some  kind of discrete gradient. This method has mean-square order 1, and for details refer to \cite{chen2014conservative,cohen2014energy,hong2011discrete}.

Next we consider the projection method with Euler-Maruyama method $\widehat{X}$ as its supporting method and choose $\nabla I(x)$ as the projection direction:
\begin{equation}\label{e:euler_p}
	\bar{X} = x + h\bar{f}(x) + \Delta\widehat{W}(h)g(x) +\lambda \nabla I(x), \quad \,\text{s.t.}\quad I(\bar{X})=I(x),
\end{equation}
where $\bar{f}=f+\frac{1}{2}g'g$ is the drift term in the It\^{o} sense.
Taking inner product on both sides of \eqref{e:euler_p} with  $\bar{\nabla}I(x,\bar{X})$, by  \eqref{dg_property} we get 
\begin{align*}
	0 	&= I(\bar{X})-I(x)=\bar{\nabla}I(x,\bar{X})^\top \cdot (\bar{X}-x)\\
	&=h\bar{\nabla}I(x,\bar{X})^\top \cdot \bar{f}(x)+ \Delta\widehat{W}(h)\bar{\nabla}I(x,\bar{X})^\top\cdot g(x) + \lambda  \bar{\nabla}I(x,\bar{X})^\top \cdot \nabla I(x).
\end{align*}
Therefore,
\begin{equation}\label{e:lam}
	\lambda = - h\frac{\bar{\nabla}I(x,\bar{X})^\top\cdot\bar{f}(x) }{ \bar{\nabla}I(x,\bar{X})^\top\cdot \nabla I(x)} - \Delta\widehat{W}(h)\frac{\bar{\nabla}I(x,\bar{X})\cdot g(x)}{ \bar{\nabla}I(x,\bar{X})^\top\cdot \nabla I(x)}.
\end{equation}
Substituting \eqref{e:lam} into \eqref{e:euler_p} leads to
\begin{equation}\label{e:euler_pdg}
	\bar{X} = x + h\bar{S}\bar{\nabla}I(x,\bar{X}) + \Delta\widehat{W}(h)\bar{T}\bar{\nabla}I(x,\bar{X}),
\end{equation}
where 
\begin{equation}\label{e:STbar}
	\begin{aligned}
		\bar{S} &= \frac{\bar{f}(x)\nabla I(x)^\top - \nabla I(x)\bar{f}(x)^\top}{ \bar{\nabla}I(x,\bar{X})^\top\cdot \nabla I(x)},\\
		\bar{T} &= \frac{g(x)\nabla I(x)^\top - \nabla I(x)g(x)^\top}{\bar{\nabla}I(x,\bar{X})^\top\cdot \nabla I(x)}.
	\end{aligned}
\end{equation}
Besides,
\begin{equation*}
	g(x)'g(x) =
	\begin{pmatrix}
		0 & -\sigma\\
		\sigma &  0
		\end{pmatrix}
	\begin{pmatrix}
		-\sigma x_2\\
		\phantom{-}\sigma x_1
	\end{pmatrix}
	= \sigma^2 x = \sigma^2 \nabla I(x).
\end{equation*}
therefore
\begin{equation*}
	\big(g(x)'g(x)\big)\cdot \nabla I(x)^ T- \nabla I(x) \cdot \big(g(x)'g(x)\big)^\top = 0,
\end{equation*}
and
\begin{equation*}
	f(x)\nabla I(x)^\top-\nabla I(x)f(x)^\top = \bar{f}(x)\nabla I(x)^\top-\nabla I(x)\bar{f}(x)^\top.
\end{equation*}
so the matrix $S$ in \eqref{e:ST} has another form
\begin{equation}\label{e:SS}
	S = \frac{\bar{f}(x)\nabla I(x)^\top-\nabla I(x)\bar{f}(x)^\top}{\nabla I(x)^\top \nabla I(x)}.
\end{equation}
Hence, $\bar{S}$ and $\bar{T}$ in \eqref{e:euler_pdg} are just the approximation of $S$ and $T$ by comparing \eqref{e:ST}, \eqref{e:SS} and \eqref{e:STbar}. Applying Theorem \ref{t:pp1p2} in methods \eqref{e:dg_kubo} and \eqref{e:euler_pdg}, we can conclude that the EulerP method \eqref{e:euler_p} is of mean-square order 1 as well.

\begin{figure}[htb]
	\centering
	\includegraphics[width=0.8\textwidth]{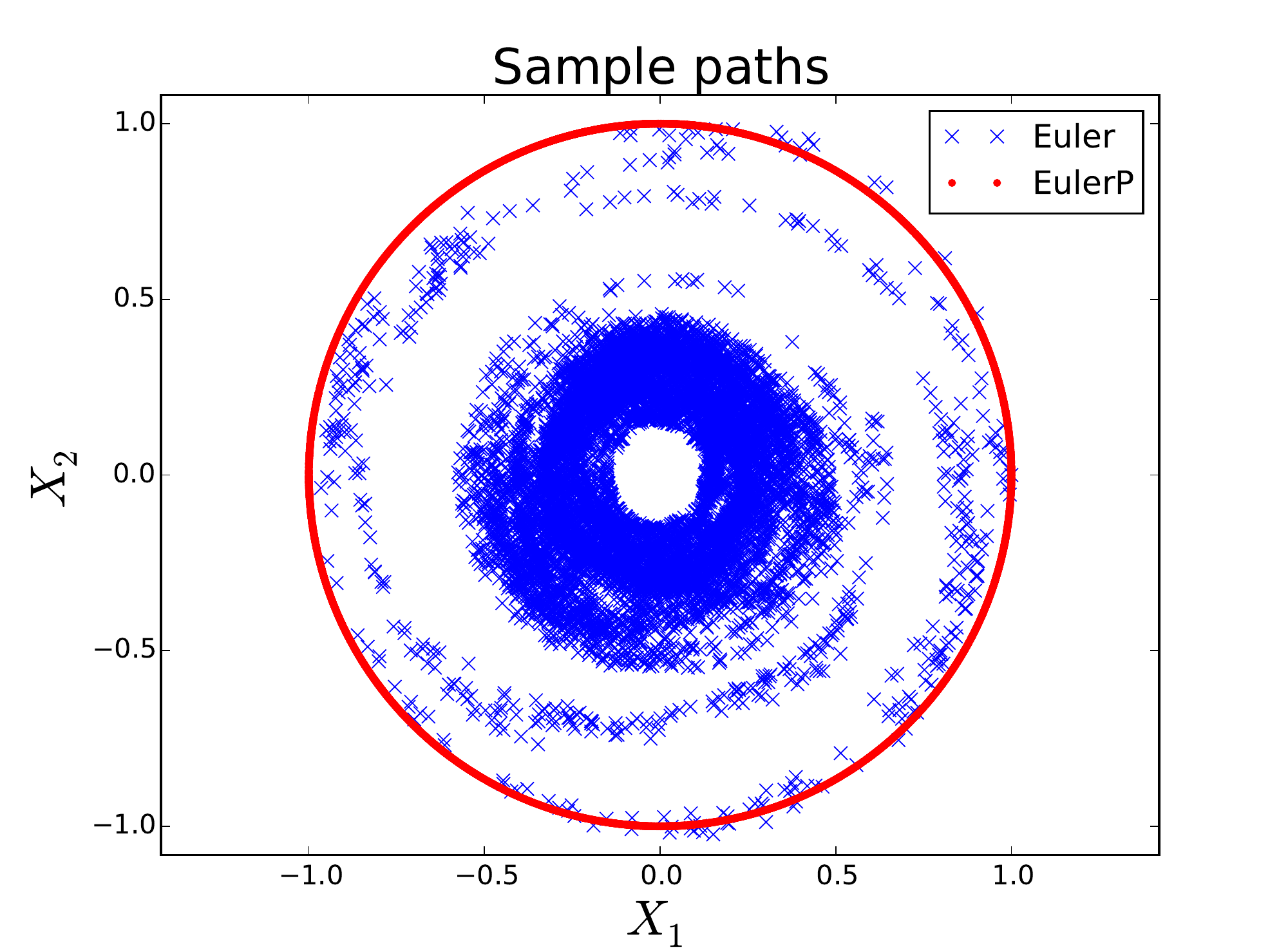}
	\caption{Numerical sample paths of Kubo oscillator produced by Euler and EulerP methods.}\label{f:kubo_path}
\end{figure}

Fig. \ref{f:kubo_path} displays two numerical sample paths of Euler method and EulerP method with $t_n\in[0,200]$ and time step-size $h=0.02$. We can observe that the numerical solution of EulerP method remains on the initial unit circle as we expect, but the normal Euler method does not share this property. The other supporting methods and their projection ones are similar except the mid-point method. Note that the midpoint method already preserve quadratic conserved quantity \cite{hong2015preservation}, so the projection step is not necessary. Errors in conserved quantity $I(x,y)=\frac{1}{2}(x^2+y^2)$ of Euler method and EulerP method are shown in Fig. \ref{f:kubo_CQ}, from which we find that the projection one preserve the conserved quantity up to the tolerance in the Newton iteration.

\begin{figure}[htbp]
	\centering
	\includegraphics[width=0.45\textwidth]{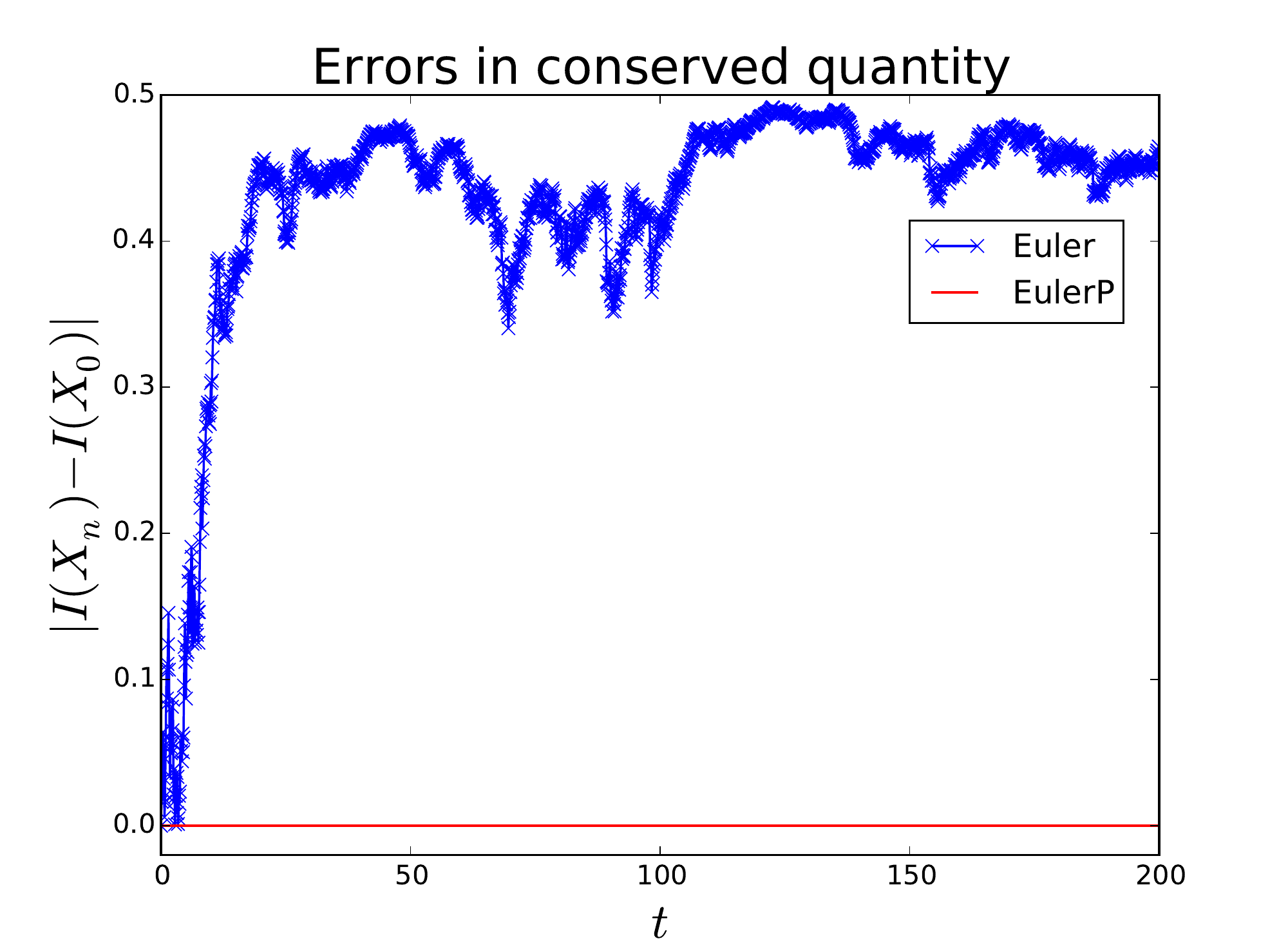}
	\includegraphics[width=0.45\textwidth]{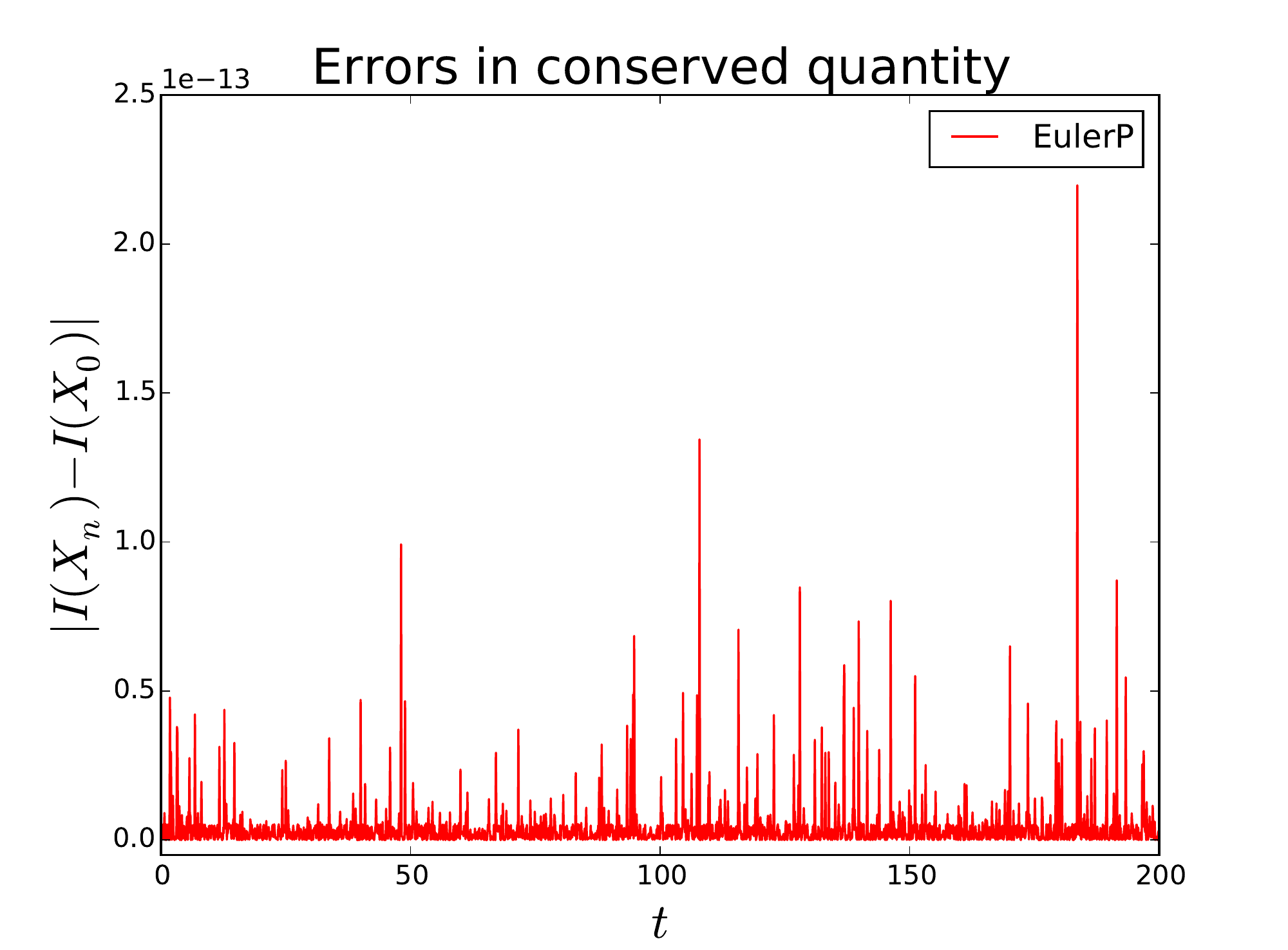}
	\caption{Errors in conserved quantity $I(x,y)=\frac{1}{2}(x^2+y^2)$ of the Euler method and EulerP method for the Kubo oscillator.}\label{f:kubo_CQ}
\end{figure}

\subsection{Example 2. Stochastic mathematical pendulum}
Consider the stochastic perturbation of a mathematical pendulum system
\begin{equation}\label{pend}
	d
	\begin{pmatrix}
		P(t) \\
		Q(t) \\
	\end{pmatrix}
	=
	\begin{pmatrix}
		-\sin\big(Q(t)\big) \\
		P(t) \\
	\end{pmatrix}
	\Big( dt + c_1\circ dW_1(t) + c_2\circ dW_2(t)\Big),
\end{equation}
with two independent Wiener process $W_1$ and $W_2$, and two real non-negative parameters $c_1$ and $c_2$.
This is also a $2$-dimensional stochastic Hamiltonian system but with a non-quadratic conserved quantity
\begin{equation}\label{e:I_pend}
	I(p,q) = \frac{1}{2}p^2 - \cos(q).
\end{equation}

\begin{table}[hbtp]
	\caption{Mean-square errors of stochastic mathematical pendulum}
	\label{table:pend}
	\begin{tabular}{@{}lccccccr@{}}
		\toprule
		$h$                 & $2^{-3}$ & $2^{-4}$ & $2^{-5}$ & $2^{-6}$ & $2^{-7}$ & $2^{-8}$ & Order \\ \midrule
		Euler               & 3.06E-01 & 1.92E-01 & 1.30E-01 & 8.89E-02 & 6.55E-02 & 4.78E-02 &  0.53 \\ \midrule
		EulerP              & 1.25E-01 & 6.88E-02 & 3.78E-02 & 2.15E-02 & 1.23E-02 & 7.48E-03 &  0.82 \\ \midrule
		Milstein            & 2.17E-01 & 9.96E-02 & 4.73E-02 & 2.27E-02 & 1.13E-02 & 5.43E-03 &  1.06 \\ \midrule
		MilsteinP\quad\quad & 1.09E-01 & 5.53E-02 & 2.73E-02 & 1.32E-02 & 6.82E-03 & 3.29E-03 &  1.01 \\ \midrule
		Mid                 & 4.72E-02 & 2.22E-02 & 1.04E-02 & 4.75E-03 & 2.39E-03 & 1.28E-03 &  1.05 \\ \midrule
		MidP                & 4.72E-02 & 2.22E-02 & 1.04E-02 & 4.75E-03 & 2.39E-03 & 1.28E-03 &  1.05 \\ \midrule
		T3/2                & 4.36E-02 & 1.44E-02 & 5.16E-03 & 1.70E-03 & 6.57E-04 & 2.27E-04 &  1.51 \\ \midrule
		T3/2P               & 3.63E-02 & 1.25E-02 & 4.34E-03 & 1.45E-03 & 5.57E-04 & 1.93E-04 &  1.51 \\ \midrule
		T2                  & 1.34E-02 & 3.68E-03 & 1.07E-03 & 2.42E-04 & 6.19E-05 & 1.56E-05 &  1.96 \\ \midrule
		T2P                 & 1.12E-02 & 2.96E-03 & 7.73E-04 & 1.89E-04 & 5.17E-05 & 1.26E-05 &  1.96 \\ \bottomrule
		                    &
	\end{tabular}
\end{table}

In this numerical test, we set $c_1 = 1$, $c_2 = 0.5$ and initial value $(p_0 , q_0)=(0.1 , 1)$. First we check the mean-square convergence order of these methods. The reference solution is computed using the T2 method with step-size $h_{\mathrm{ref}} = 2^{-14}$ and the expectation is realized using the average of $M=10000$ independent samples as previous example. The results are shown in Table \ref{table:pend} and Fig. \ref{f:pend_order}, from which we note that these methods are able to reach the theoretical mean-square order as we expect.

\begin{figure}[htbp]
	\centering
	\includegraphics[width=0.7\textwidth]{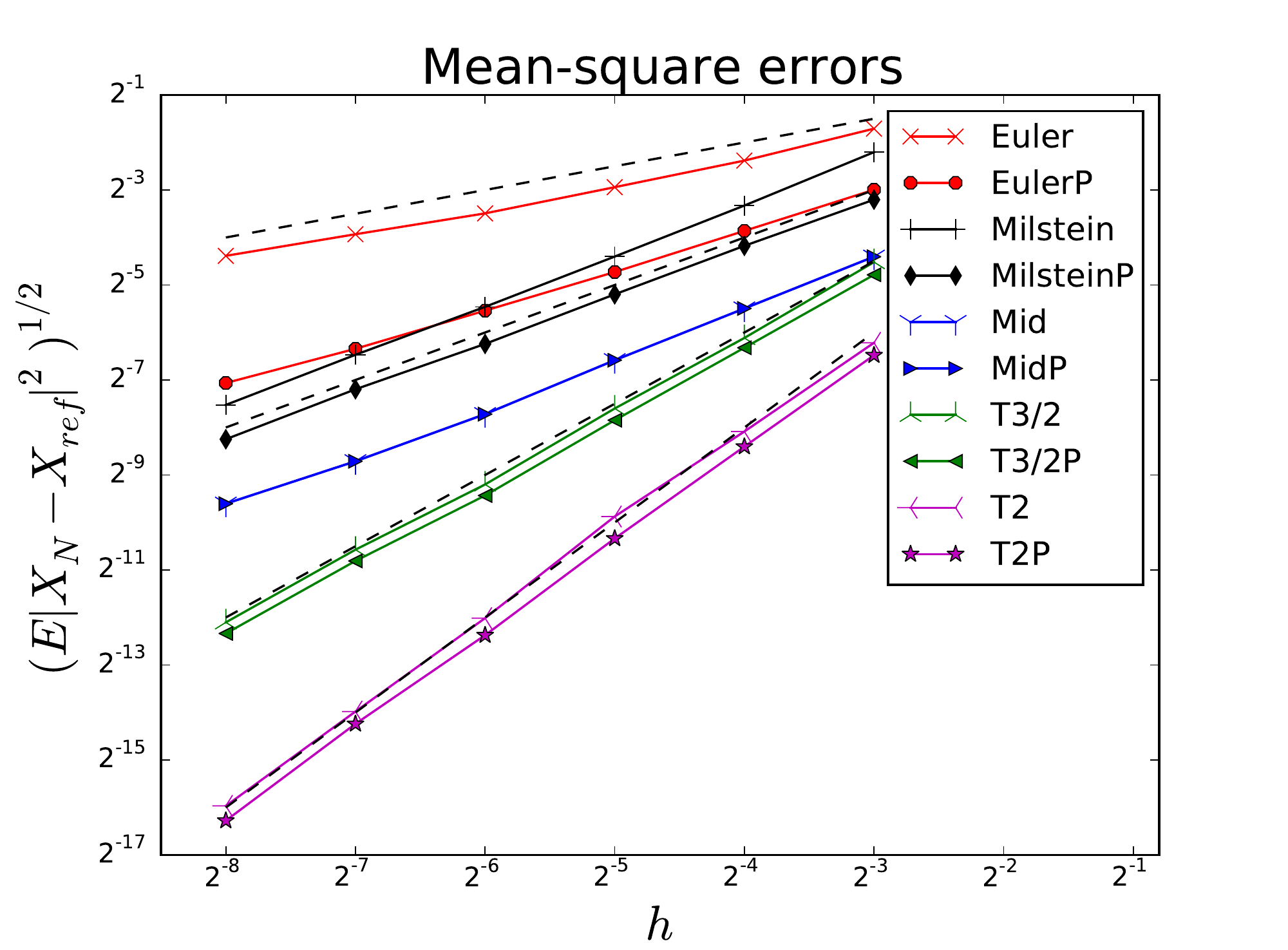}
	\caption{Mean-square errors at $T=1$ of these methods for the mathematical pendulum system. The dashed reference lines have slopes 0.5, 1, 1.5 and 2.}\label{f:pend_order}
\end{figure}

Next we consider the geometric property of the numerical solution in the phase space. Fig. \ref{f:pend_path} shows numerical sample paths of Euler and EulerP methods. In this test, we set $T=200$. As discussed before, the projection one can ensure that the numerical solution runs alone the invariant submanifold which is like a circle in this example because of the initial value. We can also check this numerically. Fig. \ref{f:pend_CQ} displays the errors in the conserved quantity $I(p,q)$ \eqref{e:I_pend} along the numerical solution by Euler and EulerP methods. The time interval is $ [0,100]$ and the time step-size is $h=0.01$. The left part compares these two methods to show the advantage of the projection one while the right one shows just the projection one individually. Thus the  conserved quantity is preserved quite well up to the accuracy of Newton iteration.

\begin{figure}[htbp]
	\centering
	\includegraphics[width=0.8\textwidth]{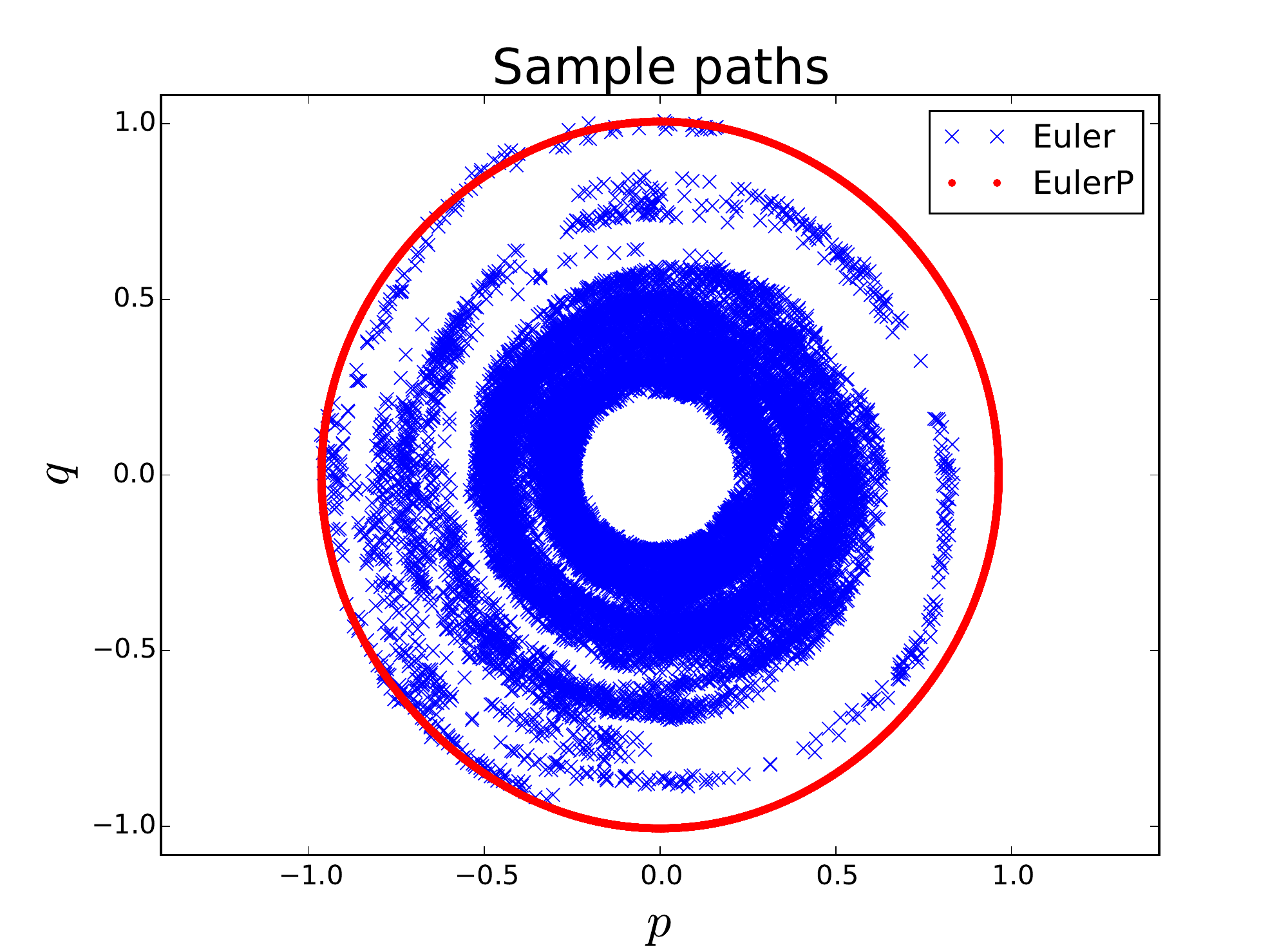}
	\caption{Numerical sample paths of the stochastic mathematical pendulum produced by Euler and EulerP methods.}\label{f:pend_path}
\end{figure}

\begin{figure}[htbp]
	\centering
	\includegraphics[width=0.45\textwidth]{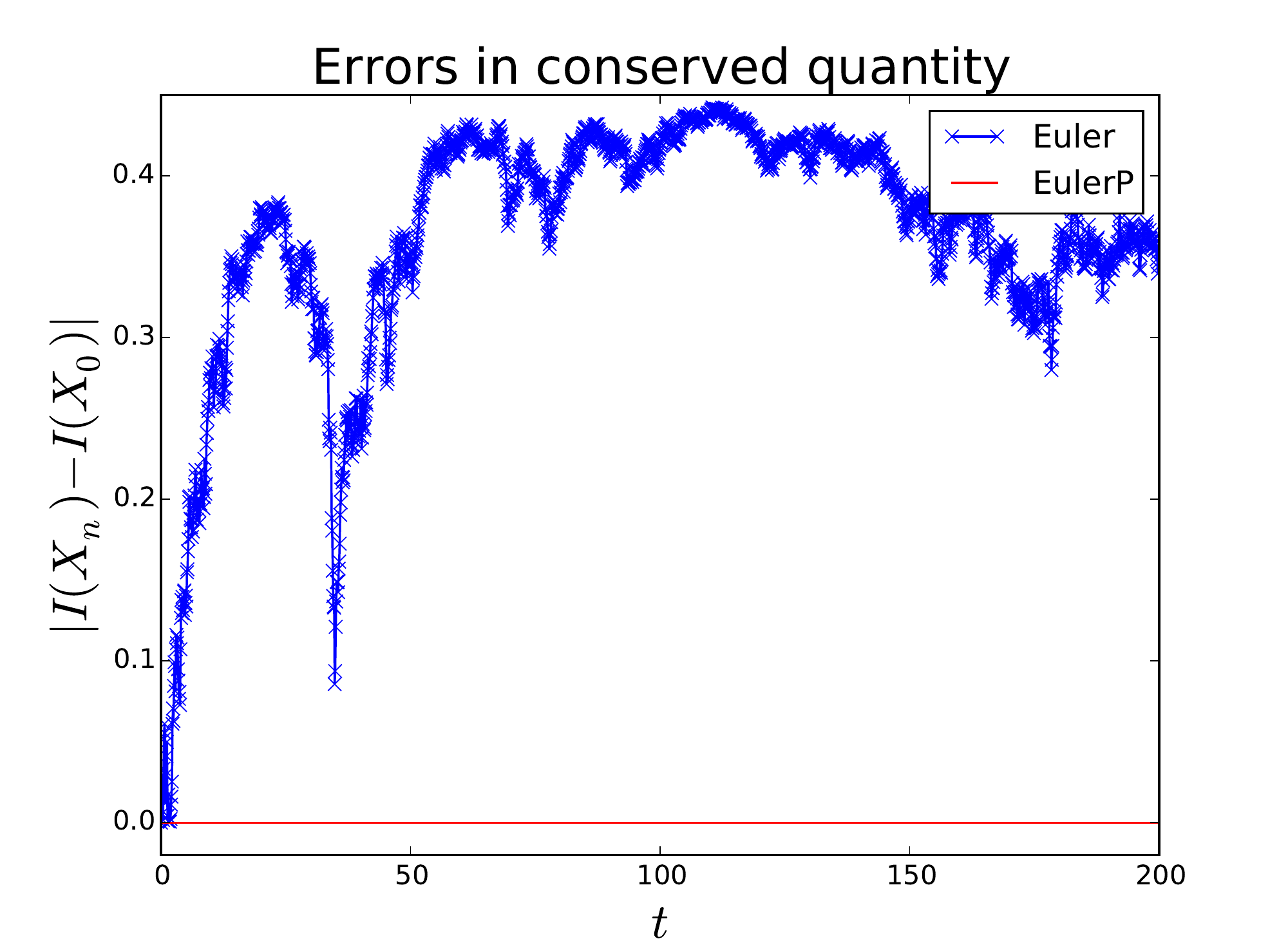}
	\includegraphics[width=0.45\textwidth]{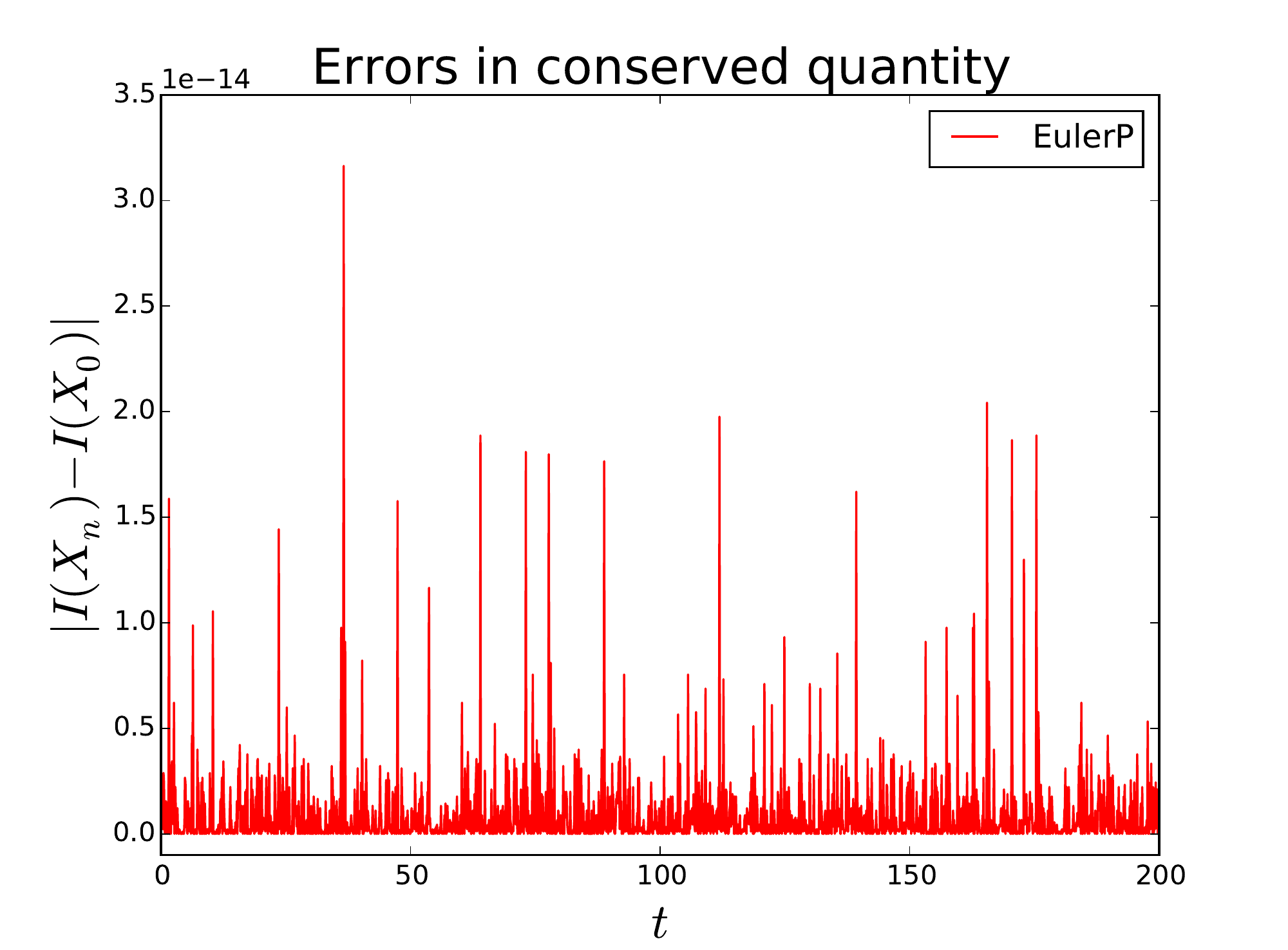}
	\caption{Errors in conserved quantity $I(p,q)=\frac{1}{2}p^2-\cos(q)$ of the Euler and EulerP methods for the stochastic pendulum system.}\label{f:pend_CQ}
\end{figure}

\subsection{Example 3. Stochastic cyclic Lotka-Volterra system}
Consider a three-dimensional SDE system
\begin{equation}\label{lk}
	d\left(
	\begin{array}{c}
		X(t) \\
		Y(t) \\
		Z(t) \\
	\end{array}
	\right) =
	\left(
	\begin{array}{c}
		X(t)\big(Z(t)-Y(t)\big) \\
		Y(t)\big(X(t)-Z(t)\big) \\
		Z(t)\big(Y(t)-X(t)\big) \\
	\end{array}
	\right) dt + c
	\left(
	\begin{array}{c}
	X(t)\big(Z(t)-Y(t)\big) \\
	Y(t)\big(X(t)-Z(t)\big) \\
	Z(t)\big(Y(t)-X(t)\big) \\
	\end{array}	\right) \circ dW(t),
\end{equation}
where $c$ is a real-valued constant.

System \eqref{lk} can be regarded as a cyclic Lotka-Volterra system of competing 3-species in a chaotic environment \cite{chen2014conservative,misawa1999conserved}. It can be easily verified that \eqref{lk} has two conserved quantities
\begin{equation}
	\begin{aligned}\label{e:lkCQ}
		I_1(x,y,z) &= x+y+z, \\
		I_2(x,y,z) &= x\cdot y\cdot z.
	\end{aligned}
\end{equation}
Note that $I_1$ is linear and even the simplest Euler method is able to exactly preserve it, but $I_2$ is not trivial here. In addition, if we just apply the discrete gradient methods or AVF methods to this model for single conserved quantity $I_1$ or $I_2$ here, the other one cannot be preserved.

Here we take $c=0.5$, and the initial value $X_0=(x_0, y_0, z_0) = (1,2,1)$. Then the  exact solution of \eqref{lk} lies in the one-dimensional submanifold
\begin{equation*}
	\mathcal{M}_{X_0} = \Big\{(x,y,z)\mid I_1(x,y,z)=x_0+y_0+z_0,\quad  I_2(x,y,z)=x_0\cdot y_0\cdot z_0\Big\},
\end{equation*}
which is a closed curve in 3-dimensional phase space.

As previous examples, we first verify the mean-square order of the proposed projection methods. Table \ref{table:lk} and Fig. \ref{f:lk_order} show the corresponding mean-square errors of these methods estimated by the average of $M=10000$ independent sample paths. We observe that the EulerP method still has mean-square order 0.5 unlike the Kubo oscillator \eqref{kubo}, so in general the projection methods cannot increase the order. Up to now, we can conclude that all the projection methods considered here are of the same mean-square order as their supporting ones.

\begin{table}[htbp]
	\caption{Mean-square errors of the stochastic cyclic Lotka-Volterra system}
	\label{table:lk}
	\begin{tabular}{@{}lccccccr@{}}
		\toprule
		$h$                 & $2^{-5}$ & $2^{-6}$ & $2^{-7}$ & $2^{-8}$ & $2^{-9}$ & $2^{-10}$ & Order \\ \midrule
		Euler               & 1.52E-01 & 9.72E-02 & 6.39E-02 & 4.35E-02 & 3.15E-02 & 2.31E-02  &  0.54 \\ \midrule
		EulerP              & 5.96E-02 & 3.75E-02 & 2.24E-02 & 1.54E-02 & 1.10E-02 & 7.80E-03  &  0.59 \\ \midrule
		Milstein            & 9.98E-02 & 4.88E-02 & 2.42E-02 & 1.18E-02 & 5.81E-03 & 2.85E-03  &  1.03 \\ \midrule
		MilsteinP\quad\quad & 5.09E-02 & 2.48E-02 & 1.25E-02 & 6.06E-03 & 3.00E-03 & 1.48E-03  &  1.02 \\ \midrule
		Mid                 & 1.51E-02 & 7.38E-03 & 3.73E-03 & 1.90E-03 & 8.62E-04 & 4.53E-04  &  1.02 \\ \midrule
		MidP                & 1.46E-02 & 7.11E-03 & 3.58E-03 & 1.82E-03 & 8.24E-04 & 4.33E-04  &  1.02 \\ \midrule
		T3/2                & 1.03E-02 & 3.42E-03 & 1.24E-03 & 4.31E-04 & 1.57E-04 & 5.75E-05  &  1.49 \\ \midrule
		T3/2P               & 9.72E-03 & 3.12E-03 & 1.19E-03 & 3.90E-04 & 1.21E-04 & 2.83E-05  &  1.65 \\ \midrule
		T2                  & 3.58E-03 & 8.72E-04 & 2.07E-04 & 5.23E-05 & 1.40E-05 & 3.46E-06  &  2.00 \\ \midrule
		T2P                 & 3.38E-03 & 8.34E-04 & 2.00E-04 & 5.35E-05 & 1.35E-05 & 3.46E-06  &  1.98 \\ \bottomrule
		                    &
	\end{tabular}
\end{table}

\begin{figure}[htbp]
	\centering
	\includegraphics[width=0.8\textwidth]{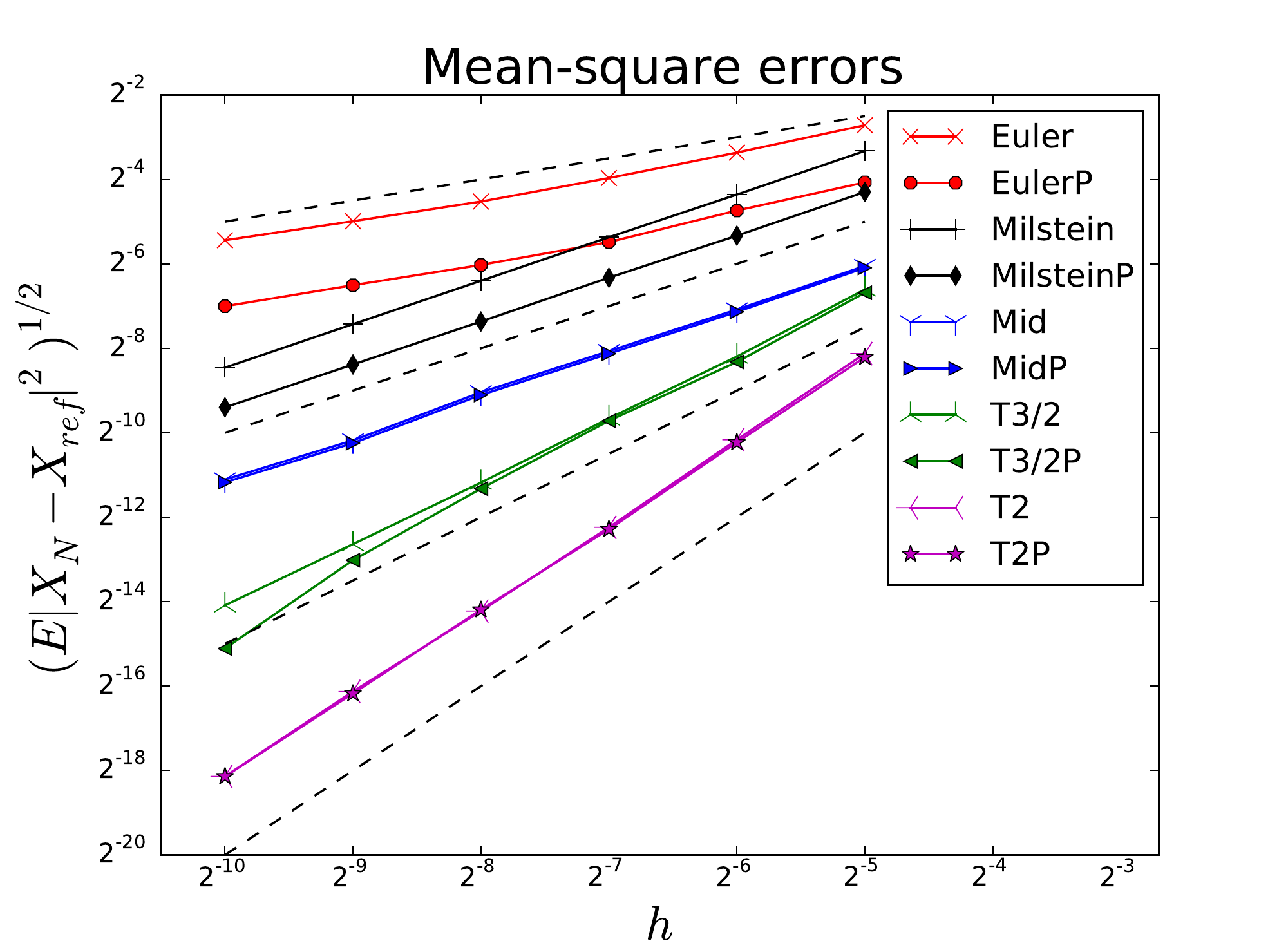}
	\caption{Mean-square errors of methods for stochastic cyclic Lotka-Volterra system. The dashed reference lines have slope 0.5, 1, 1.5 and 2 respectively. }\label{f:lk_order}
\end{figure}

For visualization purpose, Fig. \ref{f:lk_path} shows numerical trajectories in 3-dimensional phase space given by the Euler and the EulerP methods. The step-size $h$ is chosen to be 0.01. We directly observe that the numerical trajectory of the EulerP method stays exactly on the initial invariant submanifold $\mathcal{M}_{X_0}$ while the other one do not. The other higher order projection methods proposed in Sect. \ref{s:multi} possess this similar property as well. Fig. \ref{f:lk_CQ} illustrates this property more clearly in errors of conserved quantities. Since system \eqref{lk} has two conserved quantities \eqref{e:lkCQ}, the errors here are estimated path-wise as
\[
	error(I_1,I_2) = \sqrt{\Big(I_1(X_{t_n}) - I_1(X_0)\Big)^2 + \Big(I_2(X_{t_n}) - I_2(X_0)\Big)^2}.
\]
We notice in the left part of Fig. \ref{f:lk_CQ} that the mid-point method seems to behave as well as the EulerP in preserving the conserved quantities. But if we just compare these two, the projection one behaves much better than the mid-point method since $I_2(x)$ in \eqref{e:lkCQ} is not quadratic.

\begin{figure}[htbp]
	\centering
	\includegraphics[width=0.9\textwidth]{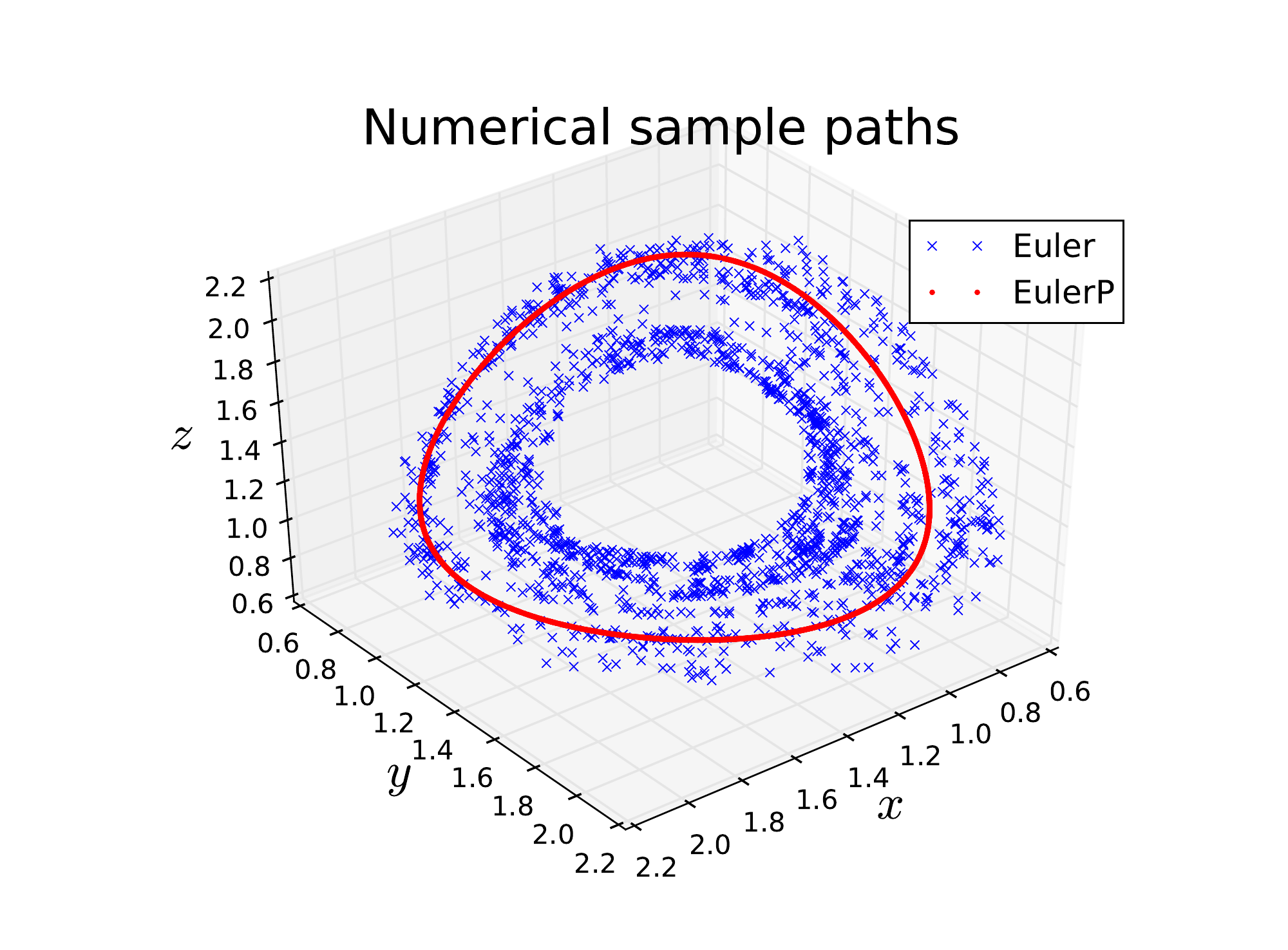}
	\caption{Numerical sample paths of the stochastic cyclic Lotka-Volterra system produced by Euler and EulerP methods.}\label{f:lk_path}
\end{figure}

\begin{figure}[htbp]
	\centering
	\includegraphics[width=0.45\textwidth]{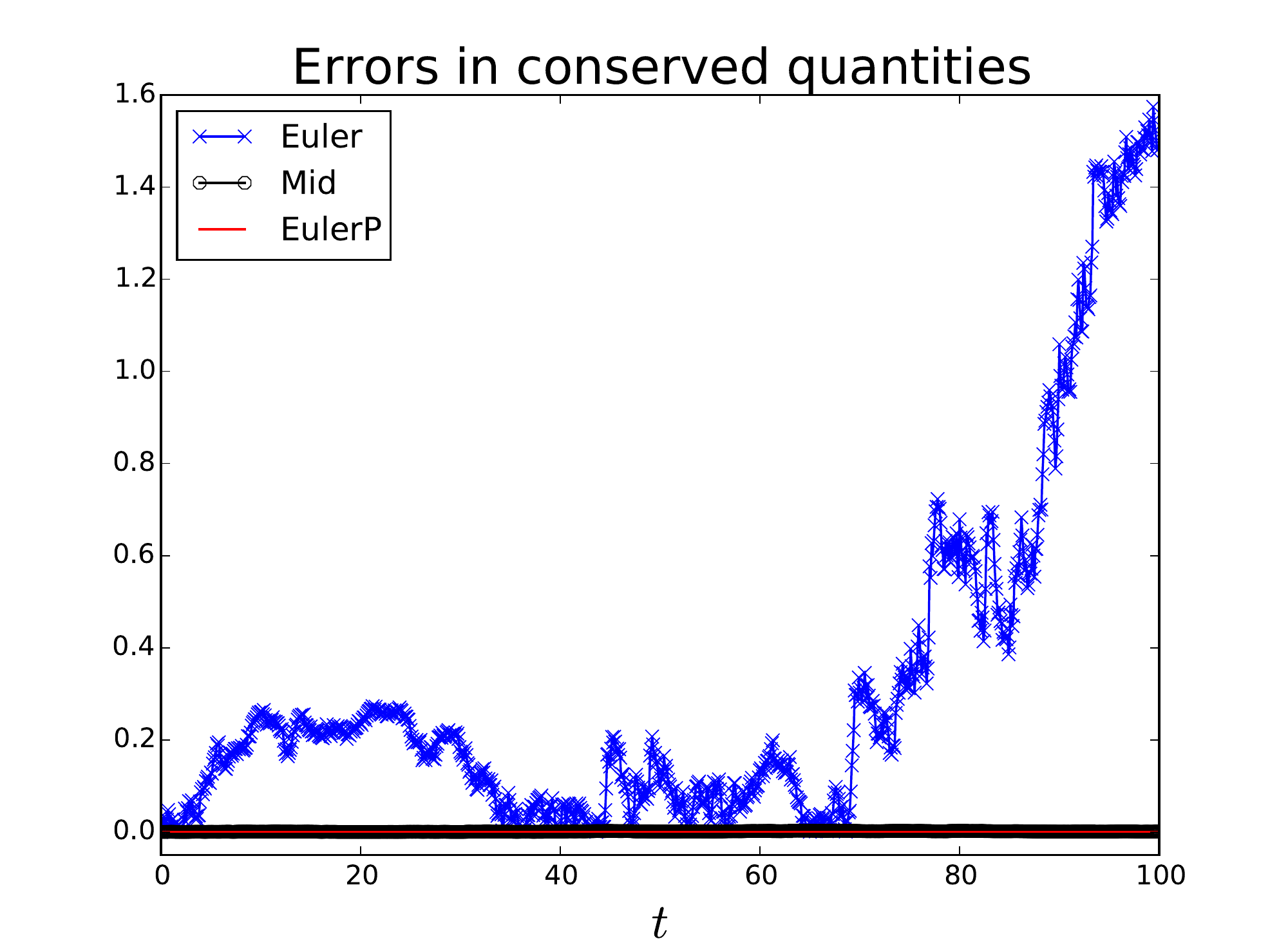}
	\includegraphics[width=0.45\textwidth]{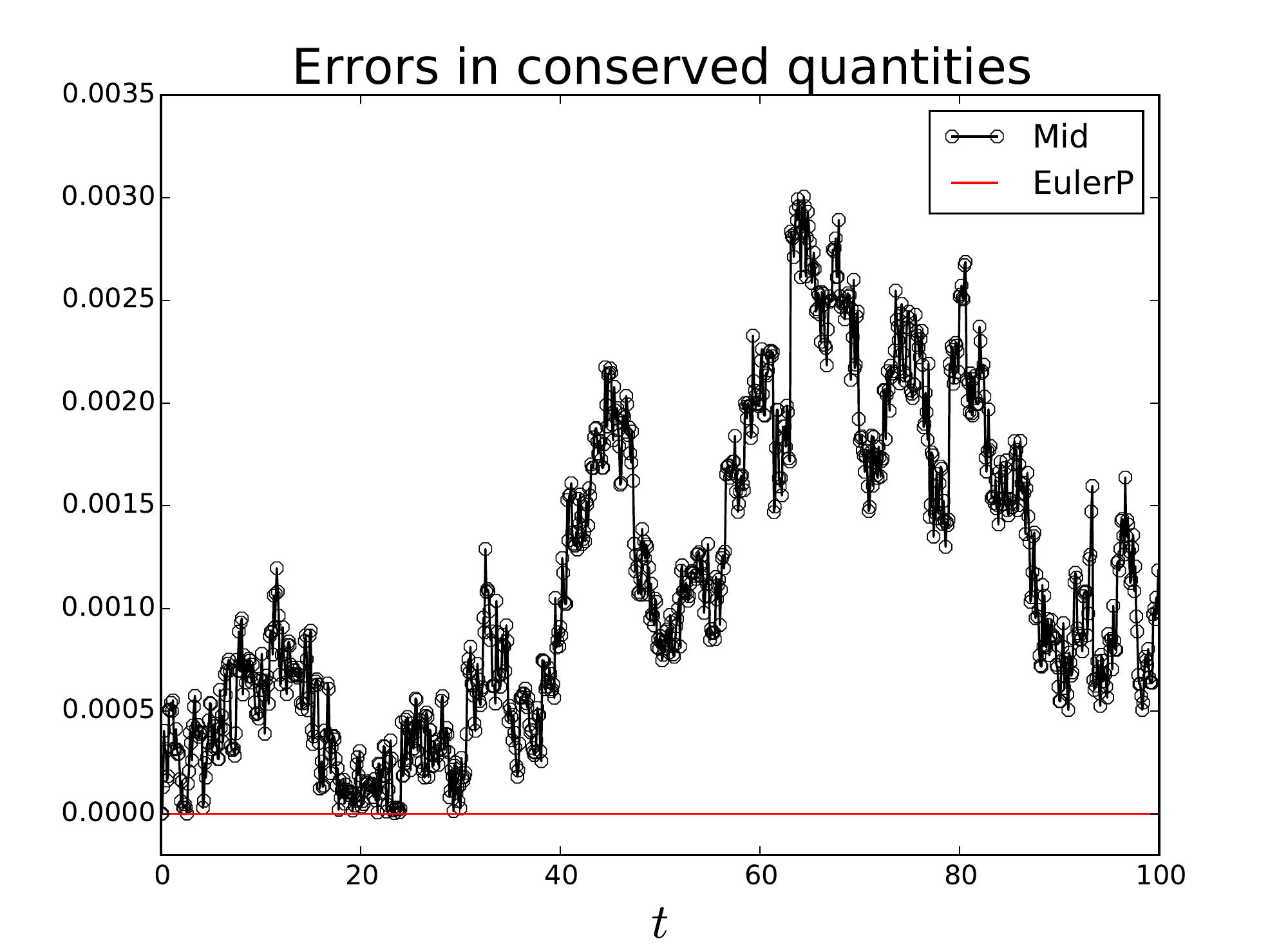}
	\caption{Errors in conserved quantities $error(I_1,I_2)$ of the Euler, EulerP and MidP methods for the stochastic  cyclic Lotka-Volterra system.}
	\label{f:lk_CQ}
\end{figure}

\section{Conclusions}
In conclusion, we generalize the projection methods to construct numerical methods which not only preserve single or multiple conserved quantities but also reach high strong order for stochastic differential equations. The obstacle of this problem lies in constructing numerical methods that preserve more than one conserved quantities if the system possesses, and increasing the order of convergence to some degree. Thus, the projection methods here couple a general supporting one-step method with a projection step so that they are able to have the same mean-square order under certain conditions. Using these projection methods, we are able to acquire high mean-square order methods which preserve multiple conserved quantities simultaneously. Eventually, three numerical examples with single or multiple conserved quantities are performed. We choose several common one-step approximations as the supporting methods before projection, which have mean-square order 0.5, 1, 1.5 and 2 respectively, and all of these projection methods for these examples preserve the conserved quantities exactly with proper mean-square order. In addition, compared with the original supporting method, the projection one can even improve the convergence order for some particular systems, such as the EulerP for Kubo oscillator which has mean-square order 1 in contrast with the Euler-Maruyama method.

\bibliographystyle{abbrv}
\bibliography{projectionRef}  

\end{document}